\documentclass{article}

\usepackage{arxiv}

\usepackage[utf8]{inputenc} % allow utf-8 input
\usepackage[T1]{fontenc}    % use 8-bit T1 fonts
\usepackage{hyperref}       % hyperlinks
\usepackage{url}            % simple URL typesetting
\usepackage{booktabs}       % professional-quality tables
\usepackage{amsfonts}       % blackboard math symbols
\usepackage{nicefrac}       % compact symbols for 1/2, etc.
\usepackage{microtype}      % microtypography
\usepackage{lipsum}
\usepackage{graphicx}
\usepackage{amssymb,amsthm,amsmath,color,makeidx}

\newcommand{\Plm}{(P_1)}

\newcommand{\intA}{\int_{\R^N}(|\nabla_A u_{\lambda,\mu}|^2+u_{\lambda,\mu}^2)dx} 
\newcommand{\intf}{\int_{\R^N}a_{\lambda}(x)| u|^q dx}
\newcommand{\intg}{\int_{\R^N}b_{\mu}(x)| u|^p dx} 
\newcommand{\jf}{J_{\lambda,\mu}} 
\newcommand{\ume}{u_{\lambda,\mu}^-}
\newcommand{\uma}{u_{\lambda,\mu}^+}

\newcommand{\J}{\mbox{$J_{\lambda}$}}

\newcommand{\HA}{\mbox{$H^1_A (\R^N)$}}
\newcommand{\HAo}{\mbox{$H_A^1(\R^N)\setminus \{0\} $}}
\newcommand{\C}{\mathbb{C}}
\newcommand{\N}{\mathbb{N}}
\newcommand{\R}{{\mathbb R}}

\newtheorem{theorem}{Theorem}[section]%[chapter]

\newtheorem{remark}[theorem]{Remark}
\newtheorem{prop}[theorem]{Proposition}
\newtheorem{lemma}[theorem]{Lemma} 
\newtheorem{definition}[theorem]{Definition}

\title{Existence, multiplicity and regularity for a Schrödinger equation with magnetic potential involving sign-changing weight function}

\author{
	de Paiva, Francisco Odair Vieira \thanks{F.O.V.P. received research grants from FAPESP 17/16108-6. } \\
	Departamento de Matem\'atica, UFSCar\\
	S\~{a}o Carlos, SP,  13560-970 Brazil\\
	\texttt{franciscoodair@gmail.com} \\
	\And
	de Souza Lima, Sandra Machado \thanks{S.M.S.L. was supported by CAPES/Brazil and the paper was completed while the second author was visiting the Departament of Mathematics of UFJF, whose hospitality she gratefully acknowledges.  }\\
	Departamento de Ciências Exatas, Biológicas e da Terra\\ INFES-UFF\\
	Santo Antônio de Pádua - RJ, 28470-000, Brazil\\
	\texttt{sandra.msouzalima@gmail.com} \\
	\And
	Miyagaki, Olimpio Hiroshi  \thanks{ O. H. M. received research grants from CNPq/Brazil 307061/2018-1, FAPEMIG CEX APQ 00063/15 and INCTMAT/CNPQ/Brazil. }\\
	Departamento de Matemática, UFJF\\
	Juiz de Fora, MG, 36036-900, Brazil\\
	\texttt {Corresponding author:{ohmiyagaki@gmail.com}} \\
}

\begin{document}
	\maketitle
	
	\begin{abstract}
	 In this paper we consider the following class of elliptic problems
	$$- \Delta_A u + u = a_{\lambda}(x) |u|^{q-2}u+b_{\mu}(x) |u|^{p-2}u ,\,\, x\in \R^N$$
	where $1<q<2<p<2^*-1= \frac{N+2}{N-2}$, $a_{\lambda}(x)$ is a sign-changing weight function, $b_{\mu}(x)$ have some aditional conditions, $u \in H^1_A(\R^N)$ and $A:\R^N \rightarrow\R^N$ is a magnetic potential. Exploring the relationship between the Nehari manifold and fibering maps, we will discuss the existence, multiplicity and regularity of solutions.
	\end{abstract}

	% keywords can be removed


\section{Introduction}

In this work we are interested in studying the existence, multiplicity of solutions for this class of elliptic problem. We will address the following concave-convex elliptic problem
$$
\left\{ \begin{array} [c]{ll}
- \Delta_A u + u = a_{\lambda}(x) |u|^{q-2}u+b_{\mu}(x) |u|^{p-2}u ,  \, \, x\in \, \,
\R^N, & \\
u \in \HA,&\\

\end{array}
\right.\leqno {\Plm}
$$
\noindent where $N\geq 3 $, $-\Delta_A =(-i\nabla+A)^2$, $1<q<2<p<2^*= \frac{2N}{N-2}$,
$a_{\lambda}(x)$ is a family of functions that can change signal, $b_{\mu}(x)$ is continuous and satisfies some additional conditions, $u:\R^N \rightarrow \C$ with
$u \in H^1_A(\R^N)$ (such space will be defined later),  $\lambda > 0$ and $\mu > 0 $ are real parameters, $A:\R^N \rightarrow\R^N$ is a magnetic potential in $L^2_{loc}(\R^N,\R^N)$. In these case we will try to show the existence of three solutions and also prove their regularity.

We will make use of the magnetic operator in which we work with the Magnetic Laplacian. In non-relativistic quantum physics, the Hamiltonian associated with a charged particle in an electromagnetic field is given by $(i\nabla -A)^2 + V$, where $A: \R^N \rightarrow \R^N$ is the potential magnetic and $ V: \R^N \rightarrow \R $  is the electrical potential. Its importance in physics was discussed in Alves and Figueiredo \cite{ClauFig} and in Arioli and Szulkin \cite{ArSz}.

The problem $(P_1)$ with $ A = 0 $ has a vast literature. We start by citing Ambrosetti, Brezis and Cerami \cite{AmbBre}, where the following problem is considered
$$
\left\{ \begin{array}[c]{ll}
- \Delta  u + u =  \lambda u^{q-1}+  u^{p-1} \, ,x\;  \mbox{in} \, \, \Omega, & \\
u> 0 \, \, \mbox{in}\, \, \Omega,&\\
u= 0 \, \, \mbox{in}\, \, \partial\Omega,&\\
\end{array}
\right.
$$
where $\Omega $ is a bounded regular domain of $\R^N$ ($N>3$),  with smooth boundary and $1<q<2<p\leq2^*$.
Combining the method of sub and super-solutions with the variational method, the authors proved the existence of a certain $\lambda_0>0$ such that there are two solutions when $\lambda\in (0,\lambda_0)$, one solution if $\lambda=\lambda_0$ and no solutions if $\lambda>\lambda_0$. 

The concave-convex problem like
$$
\left\{ \begin{array}[c]{ll}
- \Delta  u + u = \lambda f(x) u^{q-1}+  u^{p-1} \, ,x\; \mbox{in} \, \, \Omega, & \\
u> 0 \, \, \mbox{in}\, \, \Omega,&\\
u= 0 \, \, \mbox{in}\, \, \partial\Omega,&\\
\end{array}
\right.
$$
with $f\in C(\overline{\Omega})$ a sign changing function and $1<q<2<p\leq2^*$, was studied, e.g., by Wu in \cite {Wu2}. It proves that the problem has at least two positive solutions for values of $ \lambda $ small enough.

The following problem was studied by de Paiva \cite {paiva}
$$
\left\{ \begin{array} [c]{ll}
- \Delta u = a(x)u^p + \lambda b(x)u^q , x \in \Omega,&\\
u(x) = 0, x \in \partial \Omega ,&\\
\end{array}          \right.\leqno {  }
$$
with $0 < p < 1 < q \leq  2^*  -  1$, $a$ a function that can change sign and $ b \geq 0 $. It proves the existence of a certain $ \lambda^* \in (0, \infty) $ such that the above problem has a nonnegative solution whenever $ 0 <\lambda <\lambda^* $ and no solution when $\lambda  >\lambda^* $.

From this, many studies have been devoted to the analysis of existence and multiplicity of concave-convex elliptic problems in bounded domains, as we can cite Brown \cite{B}; Brown and Wu \cite{BW}; Brown and Zhang \cite{BZ2003}; Hsu \cite{Hsu1}; Hsu and Lin \cite{Hsu} and references contained in these articles.

Hsu and Lin \cite{HsuLin} consider the following equation
$$
\left\{ \begin{array}[c]{ll}
- \Delta  u + u = a(z) u^{p-1}+  \lambda h(z)u^{q-1} \, ,x\; \mbox{in} \, \, \R^N, & \\
u \,\, \in \,\, H^1(\R^N),&\\
\end{array}
\right.
$$
for $N\geq3$,  $1\leq q<2<p<2^*$, $\lambda>0$, $a$ is continuous and positive and $ h $ is positive in a positive measure set. The authors study the existence and multiplicity of solutions to this equation.
%In which case $ q=\lambda =1$ and $a(z)=1$ for all $z\in \R^N$, assuming $h$ small in some norm, with exponential decay and non-negative Zhu \cite{Zhu} and Hsu and Wang \cite{HsuWang} showed that the above equation has at least two positive solutions in a domain outside a range of $\R^N$.

Besides these we can cite Chen \cite{Chen}, Huang, Wu and Wu \cite{HWW}, who have worked similar cases in $\R^N.$

Wu in \cite{Wu}, deals with the problem
$$
\left\{ \begin{array}[c]{ll}
-\Delta  u + u = f_{\lambda }(x) u^{q-1}+  g_{\mu}(x)u^{p-1} \, ,x\; \mbox{in} \, \, \R^N, & \\
u\geq 0 \, \, \mbox{in}\, \, \R^N,&\\
u \,\, \in \,\, H^1(\R^N),&\\
\end{array}
\right.
$$
with $1<q<2<p\leq2^*$, $g_{\mu}\geq 0$ or $f_{\lambda}$ being able to change of signal, among other additional hypotheses. It seeks to show the existence of at least four solutions to the problem when $\lambda$ and $\mu$ small enough.
This is the result that we try to extend, investigating if it would be possible to obtain similar consequences when replacing the magnetic laplacian in the place of the usual Laplacian.

The first results in non-linear Schr\"{o}dinger equations, with $ A \neq 0 $ can be attributed to Esteban and Lions \cite{EstLions} in which is obteined the existence of stationary solutions for equation of the type
$$-\Delta_A+Vu =|u|^{p-2}u, u\neq 0, u\in L^2(\R^N)  , $$
 using minimization methods for the case $ V = 1 $, $ p \in (2,\infty),$ with constant magnetic field and also for the general case.

In \cite{Kurata}, Kurata showed that equation
$$\left(\frac{h}{i}\nabla -A(x)\right)^2u+V(x)u - f(|u|^2)u=0, \,\, x \, \in \R^N,   $$
with certain assumptions about the magnetic field $A$, as well as, for the potential $V$ and $f$, has at least one solution that concentrates near the set of global minimums of $V$, as $h\rightarrow 0$.

%Também já foi provado que o potencial magnético A só contribui para o fator de fase da solução da equação acima para valores bem pequenos de $h>0$. Para este caso, resultados de multiplicidade utilizando argumentos topológicos foram provados em Cingolani\cite{Cingolani}.

Chabrowski and Szulkin \cite{ChabSzul} worked with this operator in the critical case and with the electric potential V being able to change the signal. Already Cingolani, Jeanjean and Secchi \cite{Cingolani2} considered the existence of mult-peak solutions in the subcritical case.

A problem of the type
$$-\Delta_A u = \mu |u|^{q-2}u+|u|^{2^*-2}u, u \neq  0, x \in \Omega \subset \R^N ,  $$
$\mu>0$ and $2\leq q <2^*$, is treated by Alves and Figueiredo \cite{ClauFig} in which the number of solutions with the topology of $ \Omega $ is related. 

A problem using the Laplacian magnetic was studied by Alves, Figueiredo and Furtado \cite{alves.furtado} 
$$\left( -i\nabla -A \left( \frac{x}{\lambda}\right) \right) ^2 u +u=  f(|u|^{2})u ,\,\, x \in \Omega_{\lambda}=\lambda\Omega,$$
in which the set $\Omega \subset \R^N $ is a bounded domain, $ \lambda> 0 $ is an actual parameter, $A$ is a regular magnetic field and $f$ is a superlinear function with subcritical growth. For the values of $ \lambda $ sufficiently large the authors showed also the existence and multiplicity of solutions relating the number of solutions with the topology of $ \Omega $,

We did not find in the literature works dealing with the non-zero $ A $ case envolving weight function that either changes sign or in the concave-convex case. In this way, it was necessary to construct own arguments to achieve the planned results. In addition, it can be seen that with this operator we are working with complex numbers, so the classic results of regularity, for example, do not apply directly. It was necessary to make a combination of results to be able to use regularity theory.

We will use the method introduced by Nehari in 1960, which has become very useful in critical-point theory and is currently called the Nehari manifold method. The Nehari manifold is closely linked with the behavior of functions known as fibering map. The method of fibering map introduced by Drabek and Pohozaev \cite{DP} and discussed by Brown and Zhang \cite{BZ2003}, relates the functional to a real function. Information about this function leads to a simple demonstration of the result we are looking for.

In sequence we will announce the first result desired. We will work with the hypotheses that we will enunciate next.

Consider the function $a(x) \in L^{q'}(\R^N),\; q'=\frac{p}{p-q}$ and $a_{\pm} =\pm \max\{\pm a(x),0\} \neq 0$. Let us assume
$$a_{\lambda} (x)=\lambda a_+(x)+a_-(x).$$ 

\begin{description}
	\item[($A$)] $\; a(x) \in L^{q'}(\R^N),\; q'=\frac{p}{p-q}$ and exists $ \hat{c}>0$ and $r_{a_-}>0, $ such that
	$$a_-(x)>-\hat{c} \exp(- r_{a_-}|x|) \;\;\mbox{ for all  }\;\; x \in  \,\R^N. $$
\end{description}

Still, we will assume that $b_{\mu}(x)=b_1(x)+\mu b_2 (x)$, where	
\begin{description}	
	\item[($B_1$)] $\;b_1(x)>0$ is continuous in $ \R^N $, with $b_1(x)\rightarrow 1$ as $|x|\rightarrow \infty$ and exists $r_{b_1}>0$, such that
	$$1\geq b_1(x) \geq 1-c_0\exp(-r_{b_1}|x|) \;\; \mbox{for some }\;\; c_0<1 \;\; \mbox{and for all } \;\; x \in \R^N.$$
	
	\item[($B_2$)] $ \; b_2(x)>0$ is continuous in $ \R^N $, $b_2(x)\rightarrow 0$ as $|x|\rightarrow \infty$ and exists $r_{b_2}>0$, with
	$r_{b_2}<\min\{  r_{a_-}, r_{b_1},q\}$ such that
	$$b_2(x)\geq d_0 \exp(-r_{b_2}|x|) \;\; \mbox{for some }\;\; d_0<1 \;\; \mbox{and for all } \;\; x \in
	\R^N.$$
	
\end{description}
Similar hypotheses were used in \cite{Wu}.

Observe that
\begin{equation}\label{funcional.1}
\jf(u)=\frac{1}{2}\intA - \frac{1}{q} \intf - \frac{1}{p} \intg,
\end{equation}
is the functional associated with the problem $ \Plm $ and is of class $ C^1 $ in $ \HA $ as can be seen in \cite{Rab}. Also, the critical points of $\jf(u)$ are weak solutions of problem $(P_1)$. 

Consider
$$\Upsilon_0=( 2-q )^{ 2-q } \left(  \frac{p-2}{ ||a_+||_{q'} }\right)^{p-2}   \left(  \frac{S_p
}{p-q }\right)^{p-q}, \;\;\; \mbox{where} $$
\begin{equation}\label{Sp} 
S_p=\inf_{u\in H_A^1(\R^N \setminus \{0\})}
\frac{\left(\int_{\R^N}|\nabla_Au|^2+u^2dx\right)^{\frac{1}{2}}}{\left(\int_{\R^N}|u|^pdx\right)^{\frac{2}{p}}}>0.
\end{equation}

\begin{theorem}\label{teo1.1(i)}
	Assuming the hypotheses $ (A), \; (B_1) $ and $ (B_2) $, and taking $ \Upsilon_0 $ as defined above, the problem $\Plm $ has at least one solution, provided that for each $\lambda >0$ and $\mu>0$ the inequality
	\begin{equation}\label{des.do.teo}
	\lambda^{p-2}(1+\mu||b_2||_{\infty})^{2-q}< \left(\frac{q}{2}\right)^{p-2} \Upsilon_0,
	\end{equation}
	holds.
	
\end{theorem}
We can see that the  solution of $ \Plm $ will vary according to the $ \mu $ and $ \lambda $ parameters. We will dedicate a section to deal with the behavior of $ \uma $ as $ \lambda \rightarrow 0 $ and we will see what happens also as $ \mu \rightarrow \infty $. Now, adding the hypothesis that the potential is asymptotic to a constant in infinity, we can prove the existence of a second solution.

\begin{theorem}\label{teo1.1(i)}
	Suppose that the potential $A \rightarrow d$ as $|x|\rightarrow \infty$, where $d$ constant. Assuming the hypotheses $(A),\;(B_1)$ and $(B_2)$, and taking $\Upsilon_0$ as defined above, the problem $\Plm$ has at least two solutions $\uma$ and $\ume$ with $\jf(\uma)<0<\jf(\ume),$ provided that inequality (\ref{des.do.teo}) holds.
	
\end{theorem}

In the Theorem \ref{teo1.1(i)}, the existence results holds for all values of $\lambda$ and $\mu$ that satisfy the inequality (\ref{des.do.teo}). Now, if we set $\lambda$ and $\mu$ conveniently small we will obtain the result of multiplicity, obtaining the existence of at least three solutions as stated in the following theorem.

\begin{theorem}\label{teo1.1(ii)} Suppose that the potential $A \rightarrow d$ as $|x|\rightarrow \infty$, where $d$ constant. Also, suppose the assumptions $(A),\;(B_1)$ and $(B_2)$ are satisfied, and let $\Upsilon_0$ as defined above, problem $\Plm$ has at least three solutions, provided that there exist $\lambda_0 >0$ and $\mu_0>0$ such that (\ref{des.do.teo}) holds for all  $\lambda \in (0,\lambda_0]$ and $\mu \in (0,\mu_0]$. \end{theorem}

%Para the problem  em questão, os números $\lambda_0$ and $\mu_0$ independem do valor de $a_-.$ Entretanto, considerando algumas hipóteses adicionais e tomando valores de $||a_-||_{q'}$ suficientemente pequenos temos o resultados de mais uma solução. Antes de enunciar este resultado apresentaremos as seguintes hipóteses:

%$(C_1)$ $b_1(x)< 1$ em $\R^N$ em um conjunto de medida positiva;

%$(C_2)$ $r_{b_1}>2.$

%\begin{theorem}\label{teo1.2} 	Suponha que o potencial $A \rightarrow d$ com $d$ constante when $|x|\rightarrow \infty$. Assumindo as hipóteses $(A)$, $(B_1)$, $(B_2)$, $(C_1)$ and $(C_2)$ existem valores positivos de $\tilde{\lambda_0}\leq \lambda_0$, $\tilde{\mu_0}\leq \mu_0$ and $ \nu_0$ such that para $\lambda \in (0, \tilde{\lambda_0})$, $\mu \in (0, \tilde{\mu_0})$ and $||a_-||_{q'} < \nu_0$,  the problem  $\Plm$ tem pelo menos quatro soluções. \end{theorem}

We will continue to make use of Variational Methods to prove the above theorems. In addition, we will try to show the regularity results that we express below.
\begin{theorem}%\label{regul}
	Supose that $u_0 \in \HA$ is a non-zero solution of $\Plm$ with $\lambda > 0$ and $\mu > 0$. Then
\begin{description}
		\item[  (i)]  $u_0 \in C(\R^N,\C) \cap L^{\gamma}(\R^N)$ for all $2 \leq \gamma  < +\infty$ ;
		\item[  (ii)] $|u_0|$ is positive in $\R^N$.		
	\end{description}

\end{theorem}

\begin{theorem}
	If $u_0\in \HA$ is a solution of $\Plm$, then $u_0 \in L^{\infty}(\R^N)$ and $\lim_{|x|\rightarrow
		\infty}u_0(x)=0.$
\end{theorem}

We will start by defining the functional $J_{\lambda,\mu}$ associated with problem $ (P_1) $ as well as, the Nehari manifold and we will see how they relate. We will work out the relation between the Nehari manifold and the behavior of the functions in the form $ F_u: t \rightarrow \J (tu); \; \ (t> 0) $. We will also make a study of the category theory to investigate the existence of a third solution. We will use the Theory of Regularity \cite{GT}, in order to prove that the weak solutions are, in fact, classic solutions of the problem in question. We will examine the behavior of the solution in cases where $ \lambda \rightarrow 0 $ and as $ \mu \rightarrow \infty $. We will do a study to estimate the energy levels  in different parts of the Nehari manifold, which will enable us to find two different solutions for the problem. %We will also make a study of the category theory to investigate the existence of a third solution.  Por fim, trabalharemos sob mais algumas hipóteses para estimar diferentes níveis de energia e usaremos o argumento min-max de Bahri-Li a fim de mostrar que para valores bem pequenos de $||a_-||_{q'}, $ the problem  possui pelo menos quatro soluções distintas.

\section{Initial considerations}

According to Tang in \cite{TZ}, we denote by $ H_A (\R^N)$ the Hilbert space obtained by the closing of $ C_0^{\infty} (\R^N, \C) $ with following inner product
$$<u,v>_A=Re\left(\int_{\R}\nabla_Au \overline{\nabla_Av}+u\overline{v} dx\right),$$
where $\nabla_Au:=(D_1u,D_2u,...,D_Nu)$ and $D_j:=-i\partial_j-A_j(x) $, with $j=1,2,...,N$, with $A(x)=(A_1(x),...,A_N(x))$. The norm induced by this product is given by
$$ ||u||_A^2:=\left(\int_{\R}|\nabla_Au|^2+u^2dx\right) .$$
\\
Note that,
$$(-i\nabla-A)^2\psi=-\Delta\psi+2iA\nabla \psi +|A|^2\psi+i\psi \; div A  \;\; \mbox{in} \;\; \R^N$$
$$= -\Delta \psi +2iA\nabla\psi+|A|^2\psi+i\psi \; div A$$

Is proved by Esteban and Lions, \cite[Section II]{EstLions} that for all $u \in H^1_A (\R^N)$ it is worth diamagnetic inequality
$$|\nabla|u|(x)|=\left|Re\left(\nabla u \frac{\overline{u}}{|u|}\right)\right|=\left|Re\left((\nabla u-iAu)\frac{\overline{u}}{|u|}\right)\right|\leq |\nabla _Au(x)|$$

%So, if $u \in \HA $ we have that $|u|$ belongs to the usual Sobolev space $\Ho$.

\subsection{Nehari Manifold}
We will now define the Nehari manifold and the fibering map and verify its properties from the $ J_{\lambda,\mu}$ function associated with $ \Plm $. Later, we will use this information to prove in a very simple way the existence of a solution of $ \Plm $, for convenient values of $\lambda $ and $\mu$. To obtain results of existence in this case, we introduced the Nehari manifold
$$M_{\lambda, \mu}=\{u \in \HA\setminus\{0\} :\langle J'_{\lambda,\mu}(u),u\rangle =0\}$$
where $\langle \;\;,\;\;\rangle$ denotes the usual duality between $\HA^*$ and $\HA$, where $\HA^*$ is the dual space to the corresponding $ \HA $ space.

\subsection{Fibering Map}
		
We will now present the functions of the form $F_u:t\rightarrow \jf(tu);\;\;(t>0)$, we will analyze its behavior and show its relation to the Nehari manifold.

Note that the fabering map it was defined depends on $ u $, $ \lambda $ and $ \mu $, so that proper notation would be $ F_{u, \lambda, \mu} $, but in order to simplify the notation, we will only work with $ F_u $.
		
If $u \in  \HA$, we have
\begin{equation}\label{phi}
F _u(t)=\frac{t^2}{2} ||u||_A^2  -\frac{t^{q}}{q}\intf - \frac{t^{p}}{p}\intg,
\end{equation}
\begin{equation}\label{phi'}
F' _u(t)= t ||u||_A^2 -  t^{q-1} \intf -  t^{p-1} \intg,
\end{equation}
\begin{equation}\label{phi''}
F'' _u(t)= ||u||_A^2 - (q-1)t^{q-2}\intf -(p-1)t^{p-2}\intg.
\end{equation}
		
The following remark relates the Nehari manifold and the Fibering map.
		
		\begin{remark}\label{u_in_s}
			Let $ F_u $ be the application defined above and $ u \in \HA $, then:			
			\begin{description}
				\item[  (i)]  $ u \in M_{\lambda, \mu}$ if, and only if, $F_u'(1)=0$; 
				\item[  (ii)] more generally $ tu \in M_{ \lambda, \mu} $, and only if, $F'_u (t)=0$.
			\end{description}
		
		\end{remark}	
	From the previous remark we can conclude that the elements in $ M_{\lambda, \mu}$, correspond to the critical points of the Fibering map. Thus, as $ F_u (t) \in C^2(\R^+, \R) $, we can divide the Nehari manifold into three parts
		$$ M_{\lambda, \mu}^+=\{ u \in M_{\lambda, \mu}; F''_{\lambda, \mu}(1)>0  \}; $$
		$$ M_{\lambda, \mu}^-=\{ u \in M_{\lambda, \mu}; F''_{\lambda, \mu}(1)<0  \}; $$
		$$ M_{\lambda, \mu}^0=\{ u \in M_{\lambda, \mu}; F''_{\lambda, \mu}(1)=0  \}. $$
		
		Next, we will prove some properties of the Nehari manifold $ M_{\lambda,\mu}$. For this, we will need some preliminary results.
		
		\begin{lemma}\label{lema2.4}
			\begin{description}
				\item[  ($i$)] For $\displaystyle{ u \in  M_{\lambda, \mu}^+ \cup M_{\lambda, \mu}^0,}$ we have $\displaystyle  \int_{\R^N} a_{\lambda}|u|^q dx>0$.	%serve para ver que essa integral é nao nula, já sabemos que ela é positiva.   
				\item[  ($ii$)] For $ u \in M_{\lambda, \mu}^- $, we have $ \displaystyle   \int_{\R^N} b_{\mu}|u|^p dx >0$.  
			\end{description}
		
		\end{lemma}
		
		\begin{proof}\begin{description}
				\item[$(i)$] Note that for $u \in  M_{\lambda, \mu}$
				\begin{eqnarray*}
					\nonumber   F''_{\lambda, \mu}(1)   &=&  ||u||^2_A - (q-1) \int_{\R^N} a_{\lambda}|u|^q dx - (p-1) \int_{\R^N} b_{\mu}|u|^p dx \\
					\label{2.2.1}    &=&   (2-p) ||u||^2_A - (q-p) \int_{\R^N} a_{\lambda}|u|^q dx .
				\end{eqnarray*}
				In which case $u \in  M_{\lambda, \mu}^+ \cup M_{\lambda, \mu}^0,$ and it follows whence $ \int_{\R^N} a_{\lambda}|u|^q dx>0.$  
				\item[$(ii)$]For $ u \in M_{\lambda, \mu} $ we have
				\begin{eqnarray*}\label{2.3.1}
					F''_{\lambda, \mu}(1)   &=&   (2-q) ||u||^2_A - (p-q) \int_{\R^N} b_{\mu}|u|^p dx .
				\end{eqnarray*}
				In which case  $ u \in M_{\lambda, \mu}^- $ and it follows whence $ \int_{\R^N} b_{\mu}|u|^p dx>0.$   
			\end{description}
		\end{proof}

We will now present a lemma that shows that a critical point of the functional restricted to the manifold is also a critical point of the functional in the whole space.
		 
		\begin{lemma}\label{minglobal}
			Supose that $\lambda > 0$ and $\mu > 0$. If $u_0$ is a local minimum for $\jf$ in $M_{\lambda, \mu}  $ with $u_0 \notin M^0_{\lambda, \mu}$, so $\jf'(u_0)=0$ in $H^{-1}_A$.
		\end{lemma}
		
		\begin{proof}
The proof is similar to what was done in \cite[Teorema 2.3]{BZ2003}.		
		\end{proof}
		
	The lemma below shows us under conditions on $M^0_{\lambda, \mu}$ is empty. This fact is essential for the development of this work, because under such conditions we will be in the hypotheses of the previous lemma. 
		%Nesse ponto se da a diferenca entre os dois artigos, o I e o II, pois no primeiro, a g_\mu nao muda  de sinal, dai uma parte nao pode ser desprezada e ficamos com dois termos pendentes. Isso nos faz carregar a desigualdade na relacao entre lambda e mu e a uma norma por todo o artigo. No artigo II como a funcao muda de sunal o termo negativo é desprezado, dai fica mais facil a comparacao. nao precisa carregar tal desigualdade por todo o artigo, é so considderar um lambda pequeno...
		
		\begin{lemma}\label{N0}
			Let $\mu\geq 0$ and $\lambda>0$ such that
			\begin{equation}\label{lambda}
	\lambda^{p-2}(1+\mu||b_2||_{\infty})^{2-q} < \Upsilon_0.
			\end{equation}
			Then $M^0_{\lambda, \mu}=\emptyset$.
		\end{lemma}
		\begin{proof}The proof is similar to what was done in \cite[Lemma 2.2]{BW}.
			
		\end{proof}
		
		With this result we have just shown that $\mu\geq 0$ and $\lambda>0$ are such that $\lambda^{p-2}(1+\mu||b_2||_{\infty})^{2-q} < \Upsilon_0$ then 
		\begin{equation}\label{uniao}
		M_{\lambda, \mu}=M^+_{\lambda, \mu} \cup M^-_{\lambda, \mu}.
		\end{equation}

	%	\subsection{Definition of $m_{\mu,u} $ and your relation with $F_{u}(t)$}
		
		The essential nature of fibering map $F_u$ is determined by sign of $ \intf$. Consider the function $ m_{\mu,u}:\R^+\rightarrow \R$ defined by
		\begin{eqnarray}\label{m1}
		m_{\mu,u}(t) &=& t^{2-q}||u||^2_A-t^{p-q}\intg.
		\end{eqnarray}
		Observe that,
		\begin{eqnarray}\label{m'.e.f''}
		m'_{\mu,u}(t)&=& t^{-q}F''_u(t), \;\;\mbox{for}\;\; tu\in M_{\lambda,\mu}.
		\end{eqnarray}
Thus, knowing the signal of $m'_{\mu,u}$, we get the sign of $F''_{tu}(t)$, and so we can conclude if $F_{tu}$ has a local minimum point, maximum local or inflection point.

Note that, for $t>0$, $tu \in M_{\lambda,\mu}$ if and only if
		\begin{eqnarray}\label{m2}
		m_{\mu,u}(t) &=& \intf.
		\end{eqnarray}
		Deriving (\ref{m1}) we have
		\begin{eqnarray}\label{m'}
		m'_{\mu,u}(t) &=&(2-q) t^{1-q}||u||^2_A-(p-q)t^{p-q-1}\intg.
		\end{eqnarray}
		Also
		\small{
			\begin{eqnarray}\label{m''}
			m''_{\mu,u}(t) &=&(2-q)(1-q) t^{-q}||u||^2_A-(p-q)(p-q-1 )t^{p-q-2}\intg.
			\end{eqnarray}		}		
		\normalsize
		To construct an outline of $m_u$ we can analyze (\ref{m'}) and observe that as $1<q<2<p<2^*$, then $m_{\mu,u}(t) \rightarrow 0$ with values greater than zero as $t \rightarrow 0$. Also, $m_{\mu,u}(t) \rightarrow -\infty$ as $t \rightarrow \infty$, whence we get the following sketch to $m_{\mu,u}(t)$.
		
		\begin{figure}[h]
			\begin{center}
				\includegraphics[scale=0.5]{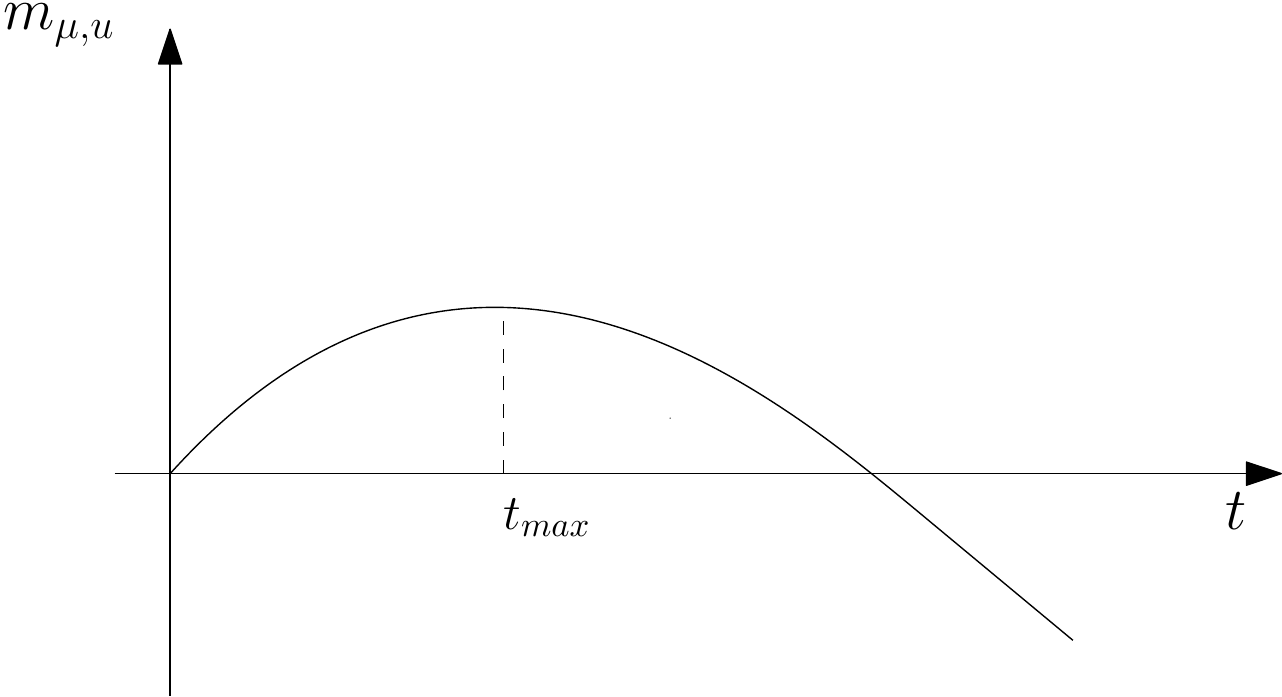}
			\end{center}
			\caption{Sketch of $m_{\mu,u}$ }
			\label{figu1}
		\end{figure}
		
		Looking at the graph and the relationship (\ref{m'.e.f''}), we have $tu \in M_{\lambda,\mu}^+$ (respectivamente, $ M_{\lambda,\mu}^-$) if and only if $m_{\mu,u}'(t)>0$ (respectively, $m_{\mu,u}'(t)<0$). So, if $u\in \HAo$, as $\intg>0$, $m_{\mu,u}(t)$ has a single critical point in $t=t_{\max}(u)$, where
		\begin{equation}\label{2.6}
		t_{\max}(u)=\left(\frac{(2-q)||u||^2_A}{(p-q)\intg}\right)^{\frac{1}{p-2}}>0.
		\end{equation}
		
		Therefore, we have two situations to study.
		\begin{description}
			\item[  ($I$)]  $\intf \leq 0.$
			
			 		In this case, we see that $\intf \leq 0$, there will be a single value $t^-(u)$ satisfying (\ref{m2}), with  $t^-(u)>t_{\max}(u)$  and such that $ m_{\mu,u}'(t^-(u))<0 .$ With this and by the relationship (\ref{m'.e.f''}), we have to for each $u \in \HA$ such that $\intf \leq 0$, there will be only one $t^-(u)$ such that $t^-(u)u \in  M_{\lambda,\mu}^-.$ We then obtain an outline for the $F_u$.
			\item[$(II)$] $ m_{\mu,u}(t_{\max}(u))> \intf > 0.$
			
		We see that if $\intf $ is such that $ m_{\mu,u}(t_{\max}(u))> \intf >0$, there will be $t^+(u)$ and $t^-(u)$ satisfying (\ref{m2}), with  $t^-(u)>t_{\max}(u) >t^+(u)$  and such that $ m_{\mu,u}'(t^-(u))<0 $ and $ m_{\mu,u}'(t^+(u))>0 .$ With this and by the relationship (\ref{m'.e.f''}), we have to for each $u \in \HA$ under these conditions, there will be $t^+(u)$ and $t^-(u)$ such that $t^+(u) u \in  M_{\lambda,\mu}^+$ and $t^-(u)u \in  M_{\lambda,\mu}^-.$ We then obtain an outline for the $F_u$.  
		\end{description}
		
Concluding, $m_{\mu,u}(t)$ is increasing in $(0,t_{\max}(u))$ and decreasing in $(t_{\max}(u),+\infty)$. Still, imposing the condition $\lambda^{p-2} (1+\mu||b_2||_{\infty})^{2-q} < \Upsilon_0,$ we get $ m_{\mu,u}(t_{\max}(u))  > \intf $.
		
	For this analysis we have just demonstrated the following result.
		
		\begin{lemma}\label{Nehari}
			For each $u \in  \HA \setminus\{0\}$ and $\mu >0$ we have
			\begin{description}
				\item[ $(i)$ ]    If $\intf \leq 0$, there is a single $t^-(u) > t_{\max}(u)$ such that $t^-(u)u \in M^-_{\lambda,\mu}$. Also, $F_{u}(t)$ is increasing in $(0,t^-(u))$, decreasing in $(t^-(u),+\infty)$ and $F_{u}(t)\rightarrow -\infty$ as $t\rightarrow +\infty$.
				\item[ $(ii)$ ] If $\intf > 0$ and $\lambda$ is such that $\lambda^{p-2}(1+\mu||b_2||_{\infty})^{2-q} < \Upsilon_0$, so there is $0<t^+(u) < t_{\max}(u)<t^-(u) $ such that $t^\pm(u)u \in M^\pm_{\lambda,\mu}$. Also, $F_{u}(t)$ is decreasing in $(0,t^+(u))$, increasing in $(t^+(u),t^-(u))$ and decreasing in $(t^-(u),+\infty)$. Furthermore, $F_{u}(t)\rightarrow -\infty$ as $t\rightarrow +\infty$.	  
			\end{description}		
		\end{lemma}
		
		We will now state another result that will be used in estimating the functional energy levels.
		
		\begin{lemma}\label{2.6(iii)}
			If $u \in  \HA \setminus\{0\}$, then
			\begin{description}
			\item[$(i)$ ]$ t^-(u)$ is a continuous function for $u\in \HA \setminus\{0\};$   
			\item[$(ii)$ ]  $ M^-_{\lambda,\mu}=\{ u\in \HA; \;\; \frac{1}{||u||_A}t^-(u)(\frac{u}{||u||_A})=1  \}.$  
			\end{description}

		\end{lemma}
		
		\begin{proof}
			The proof is similar to that made in \cite[Lemma 2.6(iii)-(iv)]{Wu}.
		\end{proof}
		
		By Lemma \ref{Nehari}, we can see that the functional is not bounded from below in $\HA$. In this case, we consider the Nehari manifold, where $\jf$ has good behavior, as will be shown in Lemma \ref{bounded} next. Still, by Lemmas \ref{N0} and \ref{Nehari}, we can see that under certain conditions of $\lambda$ and $\mu$, we have a minimizer in $M^+_{\lambda,\mu}$ and another in $M^-_{\lambda,\mu}$, whose minimum levels of energy will be denoted respectively by
		$$m^+_{\lambda,\mu}=\inf_{u \in M^+_{\lambda,\mu}}\jf (u) $$
		and
		$$m^-_{\lambda,\mu} =\inf_{u \in M^-_{\lambda,\mu}}\jf (u)  .$$
		Our next result shows that these points are well defined.
		
		\begin{lemma}\label{bounded}
			The functional $\jf$ is coercive and bounded from below in $M_{\lambda,\mu}$.
		\end{lemma}
		
		\begin{proof} The proof is similar to that made in \cite[Lemma 2.1]{Hsu}.
		\end{proof}
		
		For the next results we will need some estimates about the values of the functions in $m^\pm_{\lambda,\mu}$. To do this, consider $\lambda^{p-2}(1+\mu||b_2||_{\infty})^{2-q}<\Upsilon_0$, so by (\ref{2.2.1}) 
		\begin{equation*} 
		||u||_A^2<   \frac{p-q}{p-2}\intf  \leq  \Upsilon_0^{1/(p-2)} \frac{p-q}{p-2} S_p^{\frac{-q}{2}} ||a_+||_{L^{q'}} ||u||_A^q,
		\end{equation*}
whence
		\begin{equation}\label{u.em.nmais}
		||u||_A \leq \left( \Upsilon_0^{1/(p-2)} \frac{p-q}{p-2} S_p^{\frac{-q}{2}} ||a_+||_{L^{q'}}\right)^{1/(2-q)} ||u||_A^q,
		\end{equation}
	for all $u\in M^+_{\lambda,\mu}.$ Also, if $\lambda=0$, then (\ref{lambda}) is satisfyied, then, by Lemma \ref{Nehari}(i), $ M^+_{\lambda,\mu}=\emptyset $ and by (\ref{uniao}) we have $ M_{\lambda,\mu}= M^-_{\lambda,\mu}$ for all $\mu\geq 0$. By what has been seen, we will show the following results on the values of $m^\pm_{\lambda,\mu}$. 
		
		\begin{lemma}\label{teorema3.1}
			\begin{description}
				\item[ $(i)$  ] If $\lambda^{p-2} (1+\mu||b_2||_{\infty})^{2-q} < (\frac{q}{2})^{p-2} \Upsilon_0, $ then $m^-_{\lambda,\mu}>0$;  
				\item[ $(ii)$  ]  Let $\lambda>0$ and $\mu \geq 0$ be such that $\lambda^{p-2} (1+\mu||b_2||_{\infty})^{2-q} <\Upsilon_0,$ then $m^+_{\lambda,\mu}<0.$ In particular, if $\lambda^{p-2}(1+\mu||b_2||_{\infty})^{2-q}<(\frac{q}{2})^{p-2}\Upsilon_0,$ then
				$$m^+_{\lambda,\mu}=\inf_{M_{\lambda,\mu}}\jf(u).$$
			\end{description}

		\end{lemma}
		
		\begin{proof} The proof is similar to what was done in \cite[Theorem 3.1]{Wu}.
		\end{proof}
		
		By Lemmas \ref{Nehari} and \ref{teorema3.1}, we can conclude that for every $u \in \HAo$
		\begin{equation}
		\jf (t^-(u)u) = \max_{t\leq 0} \jf(tu),
		\end{equation}
		whenever $ \lambda^{p-2} (1+\mu||b_2||_{\infty})^{2-q}  < \left( \frac{q}{2} \right)^{p-2}\Upsilon_0,$ with $\lambda\geq 0 $ and $\mu>0.$ Moreover, if there exists $t^+(u)$, then $\jf(t^+(u)u) = \min_{0\geq t \geq t^-(u)}\jf(tu)$.  These properties are essential to show the existence of a $\Plm$ solutions. 
		
 \section{The Existence of a Solution for $\Plm$ }\label{existencia}

In this section we will show the existence of solutions to the problem $\Plm$ for $\Upsilon_0>0 $ and $\mu>0.$ First we will establish a local compactness lemma, for this, consider the following semilinear elliptical problem
$$
\left\{ \begin{array} [c]{ll}
- \Delta_A  u + u =  |u|^{p-2}u   \, \, \mbox{in} \, \,            \R^N, & \\
u \in \HA.&\\
\end{array}
\right.\leqno {(P_A)}
$$
Let $J_{\infty}(u)=\frac{1}{2}||u||^2_A-\frac{1}{p}||u||_p^p$, the functional associated with the problem ($P_A$), then $J_{\infty}$ is a functional $C^2$ in $\HA$. The Nehari manifold associated with problem $(P_A)$ is given by 
$$M_{\infty}=\{ u \in \HAo; \; J'_{\infty}(u)u=0 \}.$$ 
In this problem we can observe if $ u \in M_{\infty}$, then $||u||_A^2=||u||_p^p$. Now consider the following minimization problem
\begin{equation}\label{m.inf}
m_{\infty}= \inf_{M_{\infty}}J_{\infty}(u).
\end{equation}
To show the existence of a solution to this minimization problem let us compare our problem with the one described below. Consider
\begin{equation*}
S=\{ u\in \HA; \int_{\R^N}|u|^pdx=1 \},
\end{equation*}
and the following minimization problem
\begin{equation}\label{M.inf}
S_{\infty}= \inf_{u \in S} ||u||^2_A.
\end{equation}

% \textbf{Davenia faz essa conta para outra expressao que é uma norma em $L^2$ mais uma seminorma... isso seria equivalente a fazer the problem  com a norma em $\HA$
%  Ainda, ele fez para N=3, para estender o resultado dele basta estender o Lemma  4.1 mas ele vale para $2<p<\frac{2N}{N-2}$}.

By an adaptation of the result of D'avenia and Squassina \cite[Theorem 4.3]{davenia} for our case, there is $u_0$ satisfying the problem  (\ref{M.inf}), that is, there exists $u_0 \in \HA$ such that $S_{\infty}= \inf_{u \in S} ||u||^2_A=||u_0||^2_A$ and $\int_{\R^N}|u_0|^pdx=1 $. Comparing the levels of problems (\ref{m.inf}) and (\ref{M.inf}) we will show the following.

\begin{lemma}\label{lema.m.inf}
	Exists $\bar{u} \in \HA$ such that $m_{\infty}= \inf_{M_{\infty}}J_{\infty}(u)=J_{\infty}(\bar{u})$. 
	
\end{lemma}

\begin{proof}
	Let $u_0$ be minimization problem solution (\ref{M.inf}). Consider $\tilde{u} = \frac{u}{(\int|u|^p)^{\frac{1}{p}}}$ with $u \in \HA$. Thus, $||\tilde{u}||_p=\frac{||u||_p}{||u||_p}=1$ , giving us that  $ \tilde{u}\in S$. Then
	$$||\tilde{u}||_A^2\geq ||u_0||^2_A.$$
	Consequently, if $u \in M_{\infty}$	
	\begin{eqnarray}
	\nonumber ||u_0||^2_A &\leq & || \tilde{u}||_A^2=||u||_A^2\frac{1}{(\int|u|^p)^{2/p}}\\
	\nonumber	& =&||u||_A^2\frac{1}{||u||_A^{4/p}}=||u||_A^{\frac{2(p-2)}{p}},
	\end{eqnarray}
	whence
	\begin{equation}
	||u_0 ||^{\frac{2p}{p-2}}_A  \leq ||u||_A^2,
	\end{equation}
	for all  $u \in M_{\infty}$. We then define 
	\begin{equation}
	\bar{u}=||u_0||^{\frac{2}{p-2}}_A u_0.
	\end{equation}
	Note that $ \bar{u} \in M_{\infty}$. So,	
	\begin{eqnarray}
	\nonumber J_{\infty}(\bar{u}) & = & \frac{1}{2} ||u_0||^{\frac{4}{p-2}}_A ||u_0||_A^2-\frac{1}{p}||u_0||^{\frac{2p}{p-2}}_p\\
	\nonumber	& =& \left( \frac{1}{2}- \frac{1}{p} \right)||u_0||^{\frac{2p}{p-2}}_A\\
	\nonumber	& \leq & \left( \frac{1}{2}- \frac{1}{p} \right)||u||^2_A = J_{\infty} (u),
	\end{eqnarray}
	for all $u \in M_{\infty}$. We conclude as soon as that infimum of $J_{\infty}(u)$ is attained in $M_{\infty}$ by $\bar{u}$, that is, $ m_{\infty}= J_{\infty}(\bar{u})$.
\end{proof}	

From these considerations we will show the following result that gives us a description of a (PS) sequence of $J_{\lambda,\mu}$.
%se precisar de solucao radial, simetrica, positiva, etc citar esteban e lions.
\begin{lemma}\label{lemadecompacidade}
	Consider $\mu\geq 0$ and $\lambda>0$ such that $\lambda^{p-2}(1+\mu||b_-||_{\infty})^{2-q} < \Upsilon_0$. Let $\{u_n\}\subset \HA$ be a sequence satisfying $J_{\lambda,\mu}(u_n)=\beta +o_n(1)$ with $\beta <m^+_{\lambda,\mu}+ m_{\infty}$ and $J'_{\lambda,\mu}(u_n)= o_n(1)$ in $H^{-1}_A$ as $n\rightarrow \infty$, then there is a subsequence $\{u_n\}$ and $u_0 \in \HA$, with $u_0$ non-zero, such that $u_n=u_0+o_n(1)$ strong in $\HA$ and $J_{\lambda,\mu}(u_0)=\beta $. Also, $u_0$ is a solution of $\Plm$.
	
%(Só preciso deste ze for fazer mais de duas solucoes)	(ii) Seja $\{u_n\}\subset M^-_{\lambda, \mu}$  sequência $(PS)_{\beta}$ em $\HA$ de $J_{\lambda,\mu} $, isto é, uma sequência satisfazendo $J_{\lambda,\mu}(u_n)=\beta +o_n(1)$  e $J'_{\lambda,\mu}(u_n)= o_n(1)$ em $H^{-1}_A$ as $n\rightarrow \infty$, onde
%	$$m^+_{\lambda,\mu}+ m_{\infty}<\beta <m^-_{\lambda,\mu}+ m_{\infty},$$
%	então existe uma subsequência $\{u_n\}$ and $u_0 \in \HA$, com $u_0$ não nulo, such that $u_n=u_0+o_n(1)$ strong in $\HA$ and $J_{\lambda,\mu}(u_0)=\beta $. Moreover, $u_0$ é uma solução de $\Plm.$
\end{lemma}

\begin{proof}
	%A prova é semelhante ao feito em (Lema3.1,\cite{HWW})
%	Firstly, we remember that 	$$\Upsilon_0=( 2-q )^{ 2-q } \left(  \frac{p-2}{ ||a_+||_{q'} }\right)^{p-2}   \left(  \frac{S_p	}{p-q }\right)^{p-q}. $$
By $(A), (B_1)$ and $ (B_2)$, we obtain by a standard argument that $\{u_n\}$ is bounded  in $\HA$. Then there is a subsequence $\{u_n\}$ and $u_0 \in \HA$ such that $u_n  \rightharpoonup u_0$ weak in $\HA$ as $n\rightarrow \infty$. Taking $v_n = u_n-  u_0$, we have $v_n \rightharpoonup 0$ weak in $\HA$ as $n\rightarrow \infty$. 
	
	Denoting by $B(0,1)$ the ball centered on the origin of radius 1, we have in $ B (0,1) $ the strong convergence
	$$ \int_{B(0,1)} |u_n|^q \rightarrow \int_{B(0,1)} |u_0|^q.$$
	By the Dominated Convergence Theorem we obtain
	$$\int_{B(0,1)}a_{\lambda} ||u_n|^q-|u_0|^q| \rightarrow 0,\;\; \mbox{as}\;\;n\rightarrow \infty.$$
 	Then, by H\"{o}lder and the integrability of $a_{\lambda}$ follows
	{\small
		\begin{eqnarray*}
			\left|\int a_{\lambda} (x)(|u_n|^q-|u_0|^q)\right|&\leq& o_n(1) + \int_{ B^c(0,1)}a_{\lambda}(x)||u_n|^q-|u_0|^q|\\
			&\leq &  o_n(1) +\left( \int_{ B^c(0,1)}a_{\lambda}(x) ^{q^*}\right)^{\frac{1}{q^*}}( ||u_n||^q_p+||u_0||_p^q)\\
			&\leq& o_n(1)+ \epsilon C .
	\end{eqnarray*}}	
	As $\epsilon >0$ it is arbitrary, we have
	$$ \int a_{\lambda} (x)(|u_n|^q-|u_0|^q)=o_n(1).$$
	
	On the other hand, $(B_1)$ and $(B_2)$ and by Brezis-Lieb lemma (see \cite{willen}), we can conclude that $\mu\int b_2(x) |v_n|^p=o_n(1)$, $\int (1-b_1(x))|v_n|^p=o_n(1)$ and $\int b_{\mu}(x) (|u_n|^p - |v_n|^p-|u_0|^p)=o_n(1)$, which together with the above inequality gives us
	$$  J_{\lambda , \mu}(u_n)= J_{\infty}(v_n)+ J_{\lambda , \mu}(u_0) +o_n(1). $$
	In a similar way we obtain that $ J'_{\infty}(v_n)v_n= J'_{\lambda , \mu}(u_n)u_n- J'_{\lambda , \mu}(u_0)u_0 +o_n(1) $. By hypothesis $J'_{\lambda , \mu}(u_n) \rightarrow 0 $ strong in $\HA^{-1}$ and $u_n \rightharpoonup u_0$ weak in $\HA$ as $n \rightarrow \infty$ and so we have $J'_{\lambda , \mu}(u_0)=0$.
	
	Now, define $\delta= \limsup_{n\rightarrow \infty} \sup_{y\in \R^N} \int_{B(y,1)}|v_n|^p.$ So we have two cases:
	
	\begin{description}
		\item[ $(i)$  ] $ \delta >0$, or 
		\item[ $(ii)$  ] $ \delta = 0$. 
	\end{description}
 	Suppose that $(i)$ happen. Then there will be a sequence $\{y_n\}\subset \R^N$ such that $ \int_{B(y_n,1)} |v_n|^p\geq \frac{\delta}{2}$ and for all $n \in \N$.
	Define $\tilde{v}_n (x)=v_n(x+y_n)$. We have that $\{\tilde{v}_n  \}$ is bounded and $\tilde{v}_n  \rightharpoonup v$ weak and almost always. Making a change of variables we obtain
	$$ \int_{B(0,1)}|\tilde{v}_n |^p \geq \frac{\delta }{4}.$$
	By Sobolev embedding 
	\begin{equation}\label{v.maior.q.zero} 
	\int_{B(0,1)}|v |^p \geq \frac{\delta }{4},
	\end{equation}
    giving us $v\neq 0$. But, $ v_n \rightharpoonup 0$  
	\begin{equation} \label{vn}
	\int_{\R^N}| v_n|^p \geq\int_{B(y_n,1)}|v_n|^p \geq \frac{\delta }{2}>0.
	\end{equation} 
	See that
	$$J_{\infty}(v_n)=\frac{1}{2}\int(|\nabla_Av_n|^2+v_n^2)dx-\frac{1}{p}\int|v_n|^p dx.$$
	Like this,
	$$F_{v_n}(t)=J_{\infty}(tv_n)=\frac{t^2}{2}||v_n||_A^2-\frac{t^p}{p}||v_n||^p.$$
	For each $n \in \N$, we can get $t_n$ such that $t_nv_n \in M_{\infty}$. So we build a sequence $\{t_n\}\subset \R^N$ with $t_n\rightarrow t_0$ as $n\rightarrow \infty$, such that $t_nv_n \in M_{\infty}$, that is, such that $J'_{\infty}(t_nv_n)t_nv_n=0$. 
	See also that
	$$J'_{\infty}(v_n)v_n=||v_n||_A^2-||v_n||^p=o_n(1)$$
	and	
	\begin{equation}\label{a.i} F'_{v_n}(t)=J'_{\infty}(tv_n)v_n=t||v_n||_A^2-t^{p-1}||v_n||^p=o_n(1).
	\end{equation} 
	Therefore we have 
	\begin{equation}\label{vai.p.zero} 
	(t_n-t_n^{p-1})||v_n||_A^2=t_n(1-t_n^{p-2})||v_n||_A^2=o_n(1). 
	\end{equation}
	By (\ref{v.maior.q.zero}) we know that $||v_n||^2_A\nrightarrow 0$. Also note that
	$t_n^{2-p}=\frac{\int|v_n|^{p}}{||v_n||_A^2}\geq \frac{\delta}{2c}  .$
	With that and by (\ref{vai.p.zero}) we get that 
	$(1-t_n^{p-2})\rightarrow 0 $, giving us that $t_n\rightarrow 1.$
	%	Como $J'_{\infty}(v_n)v_n\rightarrow 0$ and $J'_{\infty}(t_nv_n)v_n\rightarrow 0$, temos que $t_n \rightarrow 1$ as $n\rightarrow \infty$, nos dando $t_0=1$.
	Now, see that $v_n \rightharpoonup 0$ weak in $\HA$ as $n\rightarrow \infty$. With this and by the fact $t_n \rightarrow 1$, we can conclude that
	$$ J_{\lambda, \mu}(u_n)=J_{\infty}(t_nv_n) +J_{\lambda, \mu}(u_0)+ o_n(1)\geq m_{\infty}+J_{\lambda, \mu}(u_0) .   $$
	Note that by hypothesis $ J_{\lambda, \mu}(u_n)=c+ o_n(1)$ with $c<m_{\infty} +m_{\lambda, \mu}^+$. By this we obtain
	$$c+ o_n(1)= J_{\lambda, \mu}(u_n)=J_{\infty}(t_n v_n) +J_{\lambda, \mu}(u_0)+ o_n(1)\geq m_{\infty}+J_{\lambda, \mu}(u_0)  ,$$
	giving us
	$$m_{\infty}+J_{\lambda, \mu}(u_0)  \leq  c+ o_n(1) <m_{\infty} +m_{\lambda, \mu}^+ + o_n(1),$$
	whence
	\begin{equation}\label{estm^+.}
	J_{\lambda, \mu}(u_0)  \leq m_{\lambda, \mu}^+ + o_n(1).
	\end{equation}
	
	We have already seen that $J'_{\lambda, \mu}(u_n)$ converges strongly to zero, whence we get $J'_{\lambda, \mu}(u_0)=0$. Thus $u_0 \in M_{\lambda, \mu}.$ Still, by Lemma \ref{N0}, $M^0_{\lambda, \mu} =  \emptyset$ and by Lemma \ref{teorema3.1} $m^+>0 $ and $m^-<0.$ Then,
	$$J_{\lambda, \mu}(u_0)\geq \inf_{M_{\lambda, \mu}}J_{\lambda, \mu}(u) = \inf_{M^+_{\lambda, \mu}}J_{\lambda, \mu}(u)=m^+,$$
	\noindent which contradicts what we have concluded in (\ref{estm^+.}).
	
	We conclude that ($ii$) occurs. In this case, $\{v_n\}$ such that $\int|v_n|^p\rightarrow 0$ if $n\rightarrow \infty$.
	As we already have $J'_{\infty}(v_n)v_n=o_n(1)$ with $J'_{\infty}(v_n)v_n=||v_n||_A^2-||v_n||_p^p$ and $\int|v_n|^p\rightarrow 0$, we conclude that $||v_n||^2\rightarrow 0$ giving us $u_n \rightarrow u_0$ strong in $\HA$. See also that $u_0 \neq 0$. In fact, note that if $u_0=0$ so $\tilde{v}_n =v_n=u_n$ and $ \int_{B(0,1)}|u_n|^p \geq \frac{\delta }{4}$, which we have already seen to be absurd.

\end{proof}

We can show that $M^{\pm}_{\lambda,\mu} $ is a manifold $C^1$ for appropriate $ \lambda $ values, as well as argued in \cite{edcarlos}. From this and the result we have just demonstrated, we are ready to address the existence of the first solution to the problem $\Plm$ in $ M^+_{\lambda,\mu}$.

\begin{prop}\label{teo3.3}
Assume that the conditions $ (A), (B_1) $ and $ (B_2) $ are satisfied with $ \lambda> 0 $ and $ \mu \geq 0 $ satisfying $ \lambda^{p-2} (1+\mu||b_2||_{\infty})^{2-q}< \Upsilon_0$, then exists $u^+_{\lambda,\mu} \in M^+_{\lambda,\mu}$ such that $J_{\lambda,\mu}(u^+_{\lambda,\mu})=m^+_{\lambda,\mu}$ and $J'_{\lambda,\mu}(u^+_{\lambda,\mu})=0$ in $H^{-1}_A$ and also $||u^+_{\lambda,\mu}||\rightarrow 0 $ as $\lambda \rightarrow 0$.
\end{prop}

\begin{proof}
	By Lemmas \ref{N0} and \ref{bounded}, we can apply a consequence of the Ekland Variational Principle%\ref{ekland.corolario}
	, to get a sequence $\{u_n\} \subset M^+_{\lambda,\mu}$ satisfying $J_{\lambda,\mu}(u_n)\rightarrow\inf_{u \in M_{\lambda,\mu}}J_{\lambda,\mu}(u)=m^+_{\lambda,\mu}$ and $J'_{\lambda,\mu}(u_n)= o_n(1)$ strong in $H^{-1}_A$ as $n\rightarrow \infty$.
	
	By Lemma \ref{lemadecompacidade}, there is a subsequence ${u_n}$ and $u^+_{\lambda,\mu} \in M^+_{\lambda,\mu}$ such that $u_n= u^+_{\lambda,\mu}+o_n(1)$ strong in $\HA$ as $n\rightarrow  \infty$. Follow that $J_{\lambda,\mu}(u^+_{\lambda,\mu})=m^+_{\lambda,\mu}$ and $J'_{\lambda,\mu}(u_n)=0$ in $H^{-1}_A$. Also, by (\ref{u.em.nmais}) and $u^+_{\lambda,\mu}\in M^+_{\lambda,\mu}$, we get that $||u^+_{\lambda,\mu}||\rightarrow 0 $ as $\lambda \rightarrow 0$.
	
\end{proof}

\section{Behavior of the first solution of $\Plm$ }

We can see that the solutions of $ \Plm $ vary according to the $ \mu $ and $ \lambda $ parameters. In Proposition \ref{teo3.3} deal with the behavior of $u^+_{\lambda, \mu}$ as $\lambda \rightarrow 0$. Now we will see what happens with $u^+_{\lambda, \mu}$ as $\mu \rightarrow \infty$. 

\begin{prop}\label{comportamento.da.primeira.sol}
	For each $\lambda >0$ such that $\lambda^{p-2}(1+\mu||b_2||_{\infty})^{2-q}<(\frac{q}{2})^{p-2}\Upsilon_0,$ and sequence $\{\mu_n\}\subset (0,\infty)$ with $\mu_n\rightarrow \infty $ as $n\rightarrow \infty $, there is a subsequence ${\mu_n}$ such that $u_{\lambda,\mu_n}^+ \rightarrow 0$ strong in $H^1_A(\R^N)$ as $n\rightarrow \infty$.
\end{prop}

\begin{proof}
	The proof is similar to what was done in \cite[Proposition 7.1(ii)]{HWW}.

\end{proof}
 
 \section{The second solution}

To treat the existence of the second $ \Plm $ solution, we need to make some considerations. Note that equation
$$- \Delta_A u + u = a_{\lambda}(x) |u|^{q-2}u+b_{\mu}(x) |u|^{p-2}u  \;\;\;\;\;\; \Plm $$
is such that $a_{\lambda}(x) \rightarrow 0 $ and $ b_{\mu}(x) \rightarrow 1$ as $|x| \rightarrow \infty$. Adding the hypothesis of $ A \rightarrow d$ with $d$ constant as $|x| \rightarrow \infty$, the problem  $\Plm$ converges at infinity for the problem
$$- \Delta_d u + u = |u|^{p-2}u.  \;\;\;\;\;\;(P_{\infty}), $$
where $-\Delta_d=(-i\nabla+d)^2$.

Thus, by a result of Ding and Liu \cite[Lemma 2.5]{DL}, $u$ is a $(P_{\infty})$ solution if and only if $v(x):=|u(x)| \in H^1$ it is a solution to the problem
$$- \Delta v + v = v^{p-1}; \;\;v>0 .  \;\;\;\;\;\; (E_{\infty}) $$
Moreover, the equations $(P_{\infty})$ and $ (E_{\infty}) $ have the same energy level, that is
$$ J_{\infty} (u )  = I_{\infty} (v ) = m_{\infty};$$
on what $ J_{\infty}$ and $I_{\infty}$ are the respective functional ones associated with the previous problems. Acording Berestycki, Lions \cite{BerLions} or Kwong \cite{KWONG}, the equation $(E_{\infty}) $ has a single solution $z_0$ symmetrical, positive and radial. By \cite[Theorem 2]{GNN}, for all  $\epsilon >0,$ exists $A_{\epsilon}, B_0$ and $C_{\epsilon}$ positive such that
\begin{equation}\label{exp1}
A_{\epsilon} \exp (-(1+ \epsilon)|x|)\leq z_0(x) \leq B_0 \exp (-|x|)
\end{equation}
and
\begin{equation}\label{exp2}
|\nabla z_0(x)| \leq C_{\epsilon} \exp (-(1- \epsilon)|x|).
\end{equation}

According Kurata \cite[Lemma 4]{Kurata}, defining $w_0=z_0 e^{-idx}$ we have $w_0$ is a solution of $ (P_{\infty}) $ , single, symmetrical, positive and radial. So we will have $ J_{\infty} (w_0 )  =m_{\infty}$. See also that $z_0=|w_0|$, which together with (\ref{exp1}) gives us the following inequalities
\begin{equation}\label{exp1w}
A_{\epsilon} \exp (-(1+ \epsilon)|x|)\leq |w_0(x)| \leq B_0 \exp (-|x|)
\end{equation}
and
\begin{equation}\label{exp2w}
|\nabla w_0(x)| \leq C_{\epsilon} \exp (-(1- \epsilon)|x|).
\end{equation}

Next, we will make some estimates about the minimum energy levels in the Nehari Manifold to prove the existence of a second solution. In order not to overload the notation, we will denote $ u_{\lambda,\mu}^+:= u^+$. Considering $ J(u^+) =m^+$, $ m^- = \inf_{u  \in M^-_{\lambda,\mu }}J_{\lambda,\mu }(u)$ and $ m_{\infty} = \inf_{u  \in M_{\infty}}J_{\infty}(u)= J_{\infty} (w_0 )$, we will make the following estimate for such energy levels.

\begin{prop}\label{prop4.1}
	For all  $\lambda>0$ and $\mu> 0$ satisfying $ \lambda^{p-2} (1+\mu||b_2||_{\infty})^{2-q}< \Upsilon_0$, we have $m^- < m^+ + m^{\infty}$.
\end{prop}

\begin{proof} Let $S^{N-1}=\{x \in \R^N; |x|=1 \}$ and consider $w_0(x) $ a critical point of $I_{\infty}(u)$ on what $I_{\infty}(w_0) = m^{\infty} $ and $m^{\infty}$ is the same as previously defined.
	Let $e \in S^{N-1}$ and define $w_k(x)= w_0(x+ke) $ for all  $k \in \N$.
	Our goal is to show that
	\begin{equation}\label{afirmacao1}
	\sup_{t>0}J(u^+ + tw_k) < m^+ + m^{\infty}, \;\; \mbox{ for }\;\; k \;\; \mbox{large enough. }
	\end{equation}
	Note that from the above statement, it will suffice to show that there is a certain $t^*$ such that $(u^+ + t^*w_k)\; \in \; M^-_{\lambda,\mu }$, to ensure the affirmation of Lemma.
	If $t \rightarrow 0$, then $ (u^+ + tw_k) \rightarrow u^+$ strong in $H^1(\R^N)$ and uniformly to $k \in \N$. Still, we have to $J(u^+)=m^+<0$ and how $m^+$ is infimun, it follows that for all  $\epsilon >0$, there will be a $ \underline{t}$, independent of $k$ such that $J(u^++t_kw_k)<m^++\epsilon$ for all  $t \in \lbrack 0, \underline{t})$ and for all  $k \in \N$, in particular, for $0<\epsilon<m_{\infty}$. Therefore $J(u^++tw_k)< m^+ + m^{\infty}$ for all  $t \in \lbrack 0, \underline{t})$ and for all  $k \in \N$. 
	
	On the other hand, $\frac{u^+}{t}+ w_k \rightarrow w_k$ strong in $ \HA$ and to uniformly $k\in \N$ as $t\rightarrow +\infty.$ This way, for all  $k\in \N,$ by  Lebesgue's dominated convergence theorem and conditions $(A)$, $(B_1)$ and $(B_2)$ we have
	\begin{eqnarray}
	J(u^++t_kw_k) &=& \frac{t^2}{2}(||w_0||^2+o_t(1))-\frac{t^q}{ q}\left( \int a_{\lambda}(x)|w_k|^q +o_t(1)\right) \\
	&&-  \frac{t^p}{ p}\left( \int b_1(x)|w_k|^q +o_t(1)\right) \\
	&&-  \frac{\mu t^p}{ p}\left( \int b_2(x)|w_k|^p +o_t(1)\right),
	\end{eqnarray}
	where $o_t(1)\rightarrow 0$ as $t\rightarrow +\infty$. We remember that $q<2<p$ and we can still observe that the last term of the above expression is always negative, so we need to analyze the behavior of $\int b_1(x)|w_k|^q$. Note that by condition $(B_1)$, $\int b_1(x)|w_k|^q\geq C>0$ for all  $k\in \N$ and some constant $ C $ independent of $ k $. Thus
	\begin{equation}\label{}
	J(u^+ + tw_k) \rightarrow  -\infty \;\; \mbox{ as } \; k \in \N \mbox{ as }\;\;t \rightarrow \infty.
	\end{equation}
	Then, for all  $\epsilon >0$, there will be $\overline{t}>\underline{t}$, independent of $k$ such that $J(u^+ + t_k w_k ) < m^+ + m^{\infty} $ for all  $ t > \overline{t}  $ and $k \in \N.$
	
	To complete the proof, it remains to show that
	\begin{equation}\label{paramostrar}
	\sup_{\underline{t}\leq t \leq \overline{t}}J(u^++tw_k)< m^+ + m^{\infty}, \;\;\; \mbox{for } k \mbox{ big enouth.}
	\end{equation}
	To prove this point note that since $ u^+ $ is a critical point of $ J (u) $, we have
	\begin{eqnarray*}
		J(u^++t_kw_k) &=& \frac{1}{2}||u^+||_A^2+ t Re\left( \int (\nabla_A u^+\overline{ \nabla_A w_k }+ u^+\overline{w_k})dx\right)+  \frac{1}{2}||tw_0 ||_A^2 \\
		&& -\frac{1}{q}\int a_{\lambda}(x) (|u^++tw_k|^q)  \\
		&&-  \frac{1}{p}\int b_1(x) |u^++tw_k|^p  - \frac{1}{p}\int\mu b_2(x) |u^++tw_k|^p \\
		&=& J(u^+) + I_{\infty}(tw_0)+J'(u^+)(tw_k)   +\frac{1}{q}\int a_{\lambda}(x) (|u^+|^q - |u^++ tw_k|^q +q|u^+|^{q-1}(tw_k)) \\
		&&+  \frac{1}{p}\int b_1(x) (|u^+|^p - |u^++ tw_k|^p +|tw_k |^p +p|u^+|^{p-1}(tw_k))  \\
		&&+   \frac{1}{p}\int(1- b_1(x))| tw_k|^p - \frac{\mu}{p}\int( b_2(x) |u^++tw_k|^p\\
		&&-   |u^+|^p -p|u^+|^{p-1}(tw_k))= m_+ + m_{\infty} + \tilde{K},
	\end{eqnarray*}
	on what $\tilde{K}=\Gamma_1+\Gamma_2+\Gamma_3+\Gamma_4$, whit
	$$ \Gamma_1= \frac{1}{q}\int a_{\lambda}(x) (|u^+|^q - |u^++ tw_k|^q +q|u^+|^{q-1}(tw_k)) ;$$
	$$ \Gamma_2= -\frac{1}{p}\int b_1(x) (|u^++ tw_k|^p-|u^+|^p  -|tw_k|^p -p|u^+|^{p-1}(tw_k)); $$
	$$ \Gamma_3=  \frac{1}{p}\int(1- b_1(x))| tw_k|^p;$$
	$$ \Gamma_4= - \frac{\mu}{p}\int b_2(x)( |u^++tw_k|^p- |u^+|^p -p|u^+|^{p-1}(tw_k)) . $$
	To show the statement (\ref{paramostrar}), we need to show that $\tilde{K}< 0$ for $k$ large enough to $t\in [\underline{t},\overline{t}]$. For this we will make some estimates about the values of each $ \Gamma_i $, for $ i = 1,2,3 $ and $ 4 $.

	% gamma 2
	We know that $(t+s)^p-t^p-s^p-pt^{p-1}s\geq 0$ for $s>0 $ and $t>0,$ then, we have
	$$(|u^+ tw_k|^p-|u^+|^p  -|tw_k|^p -p|u^+|^{p-1}(tw_k)) \geq 0.$$
	With this and because $ b_1> 0 $, we obtain
	\begin{eqnarray}\label{Gamma2}
	\Gamma_2 &=& -\frac{1}{p}\int b_1(x) (|u^+ tw_k|^p-|u^+|^p  -|tw_k|^p -p|u^+|^{p-1}(tw_k))\leq 0.
	\end{eqnarray}
	
	%Gamma 3
	To analyze $\Gamma_3$ note that using $ (B_1)$ and (\ref{exp1w})
	\begin{eqnarray}\label{gamma3}
	\nonumber  \Gamma_3 &=&\frac{1}{p}\int_{\R^N}(1- b_1(x))| tw_k|^p\leq  c_0 B_0^p \int_{\R^N} \exp(-r_{b_1}|x|)\exp(-p|x+le|)dx\\
	\nonumber    &\leq & c_0 B_0^p \int_{|x|<l} \exp(-\min \{r_{b_1},p\}(|x|+|x+le|))dx\\
	\nonumber    && +  c_0 B_0^p \int_{|x|\geq l} \exp(-\min \{r_{b_1},p\}(|x|+|x+le|)) dx\\
	\nonumber    &\leq & c_0 B_0^p l^n \int_{|x|<l} \exp(-\min \{r_{b_1},p\}l)dx+  C_0 B_0^p \exp(-\min \{r_{b_1},p\}l)\\
	& &+  C_0 B_0^p l^n \exp(-\min \{r_{b_1},p\}l) \;\;\mbox{ for } \;\; l\geq 1.
	\end{eqnarray}
	
	%gamma 1
Now, using an argument from Gidas, Ni and Nirenberg \cite{GNN} and the expansion of Taylor, we have
	\begin{eqnarray} \label{gamma1.1} 
	\nonumber \Gamma_1&= &\frac{1}{q}\int a_{\lambda}(x) (|u^+|^q - |u^++ tw_k|^q +|u^+|^q +q|u^+|^{q-1}(tw_k)) \\
	\nonumber&= &\int- a_{\lambda}(x) \left(\int_{0}^{tw_k}|u^++\eta |^{q-1}d\eta-tw_k|u^+|^{q-1} \right).
	\end{eqnarray}
	Note that $\int_{0}^{tw_k}|u^++\eta |^{q-1}d\eta=-|u^+|^{q-1}\eta|_0^{tw_k}=-tw_k|u_0|^{q-1}$. Also, like $- a_{\lambda}(x)=-\lambda a_+(x)-a_-(x)\leq |a_-|$ we have
	\begin{eqnarray} \label{gamma1.1.1} 
	\nonumber \Gamma_1&\leq & \int\left(|a_-|\int_{0}^{tw_k} ((u^++\eta)^{q-1} - (u^+)^{q-1}) d \eta \right) \\
	&\leq & \frac{\overline{t}^q}{q}\int_{\R^N}  |a_-| w_k^q.
	\end{eqnarray}
	Thus, by the condition $ (A_2) $ and by the same argument used in (\ref{gamma3}) we obtain
	\begin{eqnarray}\label{gamma1}
	\nonumber \Gamma_1 &\leq & \frac{\overline{t}^q}{q}\int_{\R^N}  |a_-| w_k^q\\
	&\leq & \hat{c}B_0^q k^n \frac{\overline{t}^q}{q} \exp(-\min\{r_{a_-},q\}k)dx \;\;\mbox{ for } \;\; k\geq 1.
	\end{eqnarray}

	%Gamma 4
Using arguments similar to those used in (\ref{gamma1.1}) with the conditions $ (B_1) $ and $ (B_2) $, by Gidas Ni Nirenberg \cite{GNN}, by the Taylor expansion and the fact that $ \mu> 0 $, we have
	\begin{eqnarray}\label{Gamma4.1}
	\nonumber\Gamma_4&= &-\frac{\mu}{p}\int b_2(x)( |u^+ +tw_k|^p- |u^+|^p -p|u^+|^{p-1}(tw_k)) \\
	& \geq &- \frac{\mu \bar{t}^p}{p} \int| b_2|w_k^p.
	\end{eqnarray}
	Using $(B_2)$ we have%\textbf{tem esse $k^n $ do determinante do jacobiano? Ele nao atrapalha pois k é grande dai some na desigualdade...}
	\small
	\begin{eqnarray}\label{Gamma4.2}
	\nonumber\int_{\R^N}| b_2|w_k^p& = &  \int_{\R^N}b_2(x-ke) w_0^p dx \geq \left(\min_{x\in B^N(0,1)}w_0^p(x)\right) \int_{B^N(0,1)}b_2(x-ke)   dx\\
	\nonumber	 & \geq &\left(\min_{x\in B^N(0,1)}w_0^p(x)\right)D_0 \exp(-r_{b_2}k).
	\end{eqnarray}
	\normalsize
	
	Now, by (\ref{Gamma2})-(\ref{Gamma4.2}) we get
	\begin{eqnarray*}
		\nonumber  J(u^++t_k w_k) &\leq &  m_+ + m_{\infty} +\Gamma_1+\Gamma_2+\Gamma_3+\Gamma_4\\
		\nonumber  &\leq &  m_+ + m_{\infty} +\hat{c}B_0^q k^n \frac{\overline{t}^q}{q} \exp(-\min\{r_{a_-},q\}k)dx\\
		\nonumber&+&C_0 B_0^p k^n \exp(-\min \{r_{b_1},p\}k)\\
		&-& \frac{\mu\bar{t}^p}{p}\left(\min_{x\in B^N(0,1)}w_0^p(x)\right)D_0 \exp(-r_{b_2}k).
	\end{eqnarray*}	
	We need $\tilde{K}=\Gamma_1+\Gamma_2+\Gamma_3+\Gamma_4<0$, and because it is the exponential that predominates, we need to $r_{b_2}k<\min\{r_{a_-},p,r_{b_1},q\}k$ , that is,
	$$r_{b_2}<\min\{r_{a_-},p,r_{b_1},q\}=\min \{r_{a_-},r_{b_1},q\},$$
which in fact happens by the hypothesis $ (B_2) $. Thus, there will be $ k_1 $ large enough, such that for all  $k>k_1,$ we will have $\tilde{K}<0$ giving us
	$$J(u^++t_kw_k) \leq  m_+ + m_{\infty}, \;\;\mbox{for all }\;\; t>0 \;\; \mbox{and}\;\;k>k_1 .$$	
	
	Finally, as we have already seen, we need to show that there is a certain $t^*$ such that $(u^+ + t^*w_k)\; \in \; M^-_{\lambda,\mu }$, to ensure the affirmation of Lemma.
	Let
	$$U_1=\left\{ u\in \HA ; \frac{1}{||u||_A}t^-\left(\frac{u}{||u||_A}\right)>1\right\} \; \cup \; \{0\};$$
	$$U_2=\left\{ u\in \HA ; \frac{1}{||u||_A}t^-\left(\frac{u}{||u||_A}\right)<1\right\}.$$
	Thus $ M^-_{\lambda,\mu}$ separates $ \HA $ into two related components $ U_1 $ and $ U_2 $. Further, from what we have seen in Lemma \ref{2.6(iii)} it follows that $ \HA \setminus M^-_{\lambda,\mu } = U_1 \cup U_2.  $
	If $ u\in M^+_{\lambda,\mu }$, then $t^+(u)=1$ and by item (ii) of Lemma \ref{Nehari}, we shall have
	\begin{equation}\label{tmax}
	1< t_{max,\mu}(u)< t^-(u).
	\end{equation}
	as $t^-(u)=\frac{1}{||u||_A }t^-(\frac{u}{||u||_A })$,follow that $ M^+_{\lambda,\mu } \subset U_1$.
	We affirm that there exists $ t_0> 0 $ such that $u^+_{\lambda,\mu}+t_0w_k \;\; \in\;\; U_2$. 
	
	To show this, let us first see that there exists a constant $ c> 0 $ such that
	$0<t^-(\frac{ u^+ +t w_k}{||u^+ +t w_k||_A})<c$ for each $t\geq 0$.  In fact, suppose the opposite to be absurd, that is, that there is a sequence $\{ t_n \}$ such that $t_n \rightarrow \infty$ and $t^-(\frac{ u^+ +t_n w_k}{||u^+ +t_n w_k||_ A }) \rightarrow \infty$ as $n   \rightarrow \infty$. Calling $v_n= \frac{ u^+ +t_n w_k}{||u^+ +t_n w_k||_A},$ temos $||v_n||_A=1$. As $t^-(v_n)v_n \in M^-_{\lambda,\mu } ,$ we have $||v_n||_A=1$. As $t^-(v_n)v_n \in M^-_{\lambda,\mu } ,$using Lebesgue's Dominated Convergence Theorem,
	\begin{eqnarray}
	\nonumber  \int_{\R^N}b_{\mu}v_n^p dx &=& \int_{\R^N}b_{\mu}\left(   \frac{ u^+ +t_n w_k}{||u^+ +t_n w_k||_A}\right) ^p dx   \\
	\nonumber   &=& \left(\frac{1}{||\frac{u^+}{t^n} +  w_k||_A}\right)^p   \int_{\R^N}b_{\mu}\left(  \frac{u^+}{t^n} +  w_k\right) ^p dx  \\
	\nonumber   & \rightarrow &   \frac{ \int_{\R^N}b_{\mu} w_k^p dx}{|| w_k||_A^p}\;\;\mbox{as}\;\;  n   \rightarrow \infty .
	\end{eqnarray}
	In this way we obtain
	\begin{eqnarray}
	\nonumber  J(t^-(v_n)v_n) &=&\frac{[t^-(v_n) ]^2}{2}-\frac{[t^-(v_n) ]^q}{q} \int_{\R^N}a_{\lambda}v_n^q dx\\
	\nonumber   &&-\frac{[t^-(v_n) ]^p}{p} \int_{\R^N}b_{\mu}v_n^p dx \rightarrow  -\infty  \;\;\mbox{as}\;\;  n   \rightarrow \infty ,
	\end{eqnarray}
but this contradicts the fact that functional is bounded from below in the Nehari manifold. Thus, we conclude that $ 0 <t^- (v_n)<c $ for a positive $ c $ constant.

Define
	$$t_0=\left(  \frac{p-2}{2p\alpha_{\infty} } c^2 - ||u^+||^2_A \right)^{\frac{1}{2}}+1.$$
	Note that
	\begin{eqnarray}
	\nonumber ||u^+ +t_0 w_k||_A^2  &=& ||u^+  ||_A^2 +t_0^2|| w_k||_ A^2 + o(1)\\
	\nonumber    &>& ||u^+||_A^2  + |c^2-|| u^+||_ A^2| + o(1)\\
	\nonumber   & >& c^2+ o(1)> \left[t^-\left( \frac{u^+ +t_0 w_k }{||u^+ +t_0 w_k||_A }\right)\right]^2  \;\;\mbox{as}\;\;  k   \rightarrow \infty .
	\end{eqnarray}
	Thus, there will be $k_2>k_1$ such that for all  $k>k_2,$
	$$\frac{ 1}{ ||u^+ +t_0 w_k||_A }t^-\left( \frac{u^+ +t_0 w_k }{||u^+ +t_0 w_k||_A }\right)<1$$
	that is, $u^+ +t_0 w_k \in U_2 $. 
	Now define a path $\gamma_l(s)= u^+ + st_0w_k $ to $s \in [0,1].$ Thus
	$$ \gamma_l(0)= u^+ \;\; \in \;\; U_1, \;\;\; \mbox{and} \;\;\; \gamma_l(s)= u^+ +  t_0w_k \;\; \in \;\; U_2.$$
	Since $\frac{1}{||u||_A }t^-(\frac{u}{||u||_A })$ is a continuous function for non zero values of $u$ and being $\gamma_l([0,1])$ a conected path, it follows that there will be a certain $s_l \in (0,1)$ such that $ u^+ + s_lt_0w_k \in M^-_{\lambda,\mu } $, as we wanted to proof.
	
\end{proof}

As was previously said $ M^{\lambda, \mu}$
is a $ C^1 $ range for appropriate $ \lambda$ values, as argued in \cite{edcarlos}. From this and the result we have just demonstrated, we are ready to deal with the existence of the second solution of the problem $ \Plm $ in $ M^-_{\lambda,\mu}$.

\begin{prop} \label{teo4.2}
	For each $\lambda >0 $ and $\mu\geq 0$, with $\lambda^{p-2}(1+\mu||b||_{\infty})^{2-q}< \Upsilon_0,$ exists $u^-_{\lambda,\mu} \in M^-_{\lambda,\mu}$ such that $J_{\lambda,\mu}(u^-_{\lambda,\mu})=m^-_{\lambda,\mu}$ and $J'_{\lambda,\mu}(u^-_{\lambda,\mu})=0$ in $H^{-1}_A$.
\end{prop}

\begin{proof}
	By Lemmas \ref{N0}, \ref{bounded} and by Variational Ekland Principle there will be a sequence $\{u_n\} \subset M^-_{\lambda,\mu}$ satisfying $J_{\lambda,\mu}(u_n)\rightarrow\inf_{u \in M_{\lambda,\mu}}J_{\lambda,\mu}(u)=m^-_{\lambda,\mu}$ and $J'_{\lambda,\mu}(u_n)= o_n(1)$ strong in $H^{-1}_A$ as $n\rightarrow \infty$.
	As $m^-_{\lambda,\mu}<m^+_{\lambda,\mu}+m_{\infty} $, by Lemma  \ref{lemadecompacidade}(ii), exist a subsequence ${u_n}$ and $u^-_{\lambda,\mu} \in M^-_{\lambda,\mu}$ such that $u_n= u^-_{\lambda,\mu}+o_n(1)$ strong in $\HA$ as $n\rightarrow  \infty$. Follow that $J_{\lambda,\mu}(u^-_{\lambda,\mu})=m^-_{\lambda,\mu}$ and $J'_{\lambda,\mu}(u_n)=0$ in $H^{-1}_A$.
\end{proof}

With this result we can then conclude the proof of the Theorem \ref{teo1.1(i)}.

\begin{proof}
	[Proof of the Theorem \ref{teo1.1(i)}] By Theorems \ref{teo3.3} and \ref{teo4.2} we can conclude that the $ \Plm $ problem has at least two solutions $\uma$ and $\ume$ with $\jf(\uma)<0<\jf(\ume).$
	
\end{proof} 

 \section{ Third Solution}

\subsection{ Some considerations}
To get the third solution of the $ \Plm $ problem, we will need some results which is done next. For this, we highlight the set defined below for $\lambda=0$ and $\mu=0$
$$ M^-_{a_0,b_0}=\{u \in \HA\setminus\{0\} :\langle J'_{a_0,b_0}(u),u\rangle =0\}$$
where 
\begin{eqnarray*}\label{funcional.0}
	J_{a_0,b_0}&=&\frac{1}{2}\intA - \frac{1}{q} \int a_0(x)|u|^qdx - \frac{1}{p} \int b_0(x)|u|^pdx\\
	&=&\frac{1}{2}\intA - \frac{1}{q} \int a_-(x)|u|^qdx - \frac{1}{p} \int b_1(x)|u|^pdx.
\end{eqnarray*}

\begin{lemma}\label{5.1.a}
	We have
	$$\inf_{u \in M^-_{a_0,b_0}}J_{a_0,b_0}(u)=\inf_{u \in M^{\infty}}J_{\infty}(u)=m^{\infty}.$$
\end{lemma}

\begin{proof}
	Let $w_k$ be as defined above. Because we are working with $ \lambda = 0 $, we have $a(x)=\lambda a_+(x)+a_-(x)=a_-(x)<0$ whence $\int_{\R^N} a_-| t^-(w_k)w_k|^qdx\leq 0,$ hence by Lemma \ref{Nehari}(i) there is only one $t^-(w_k)>\left(\frac{2-q}{p-q}\right)^{\frac{1}{p-2}}$ such that $t^-(w_k)w_k \in  M^-_{a_0,b_0}$ for all $k>0,$ that is, $ J'_{a_0,b_0} (t^-(w_k)w_k)=0, $ giving us
	\begin{equation}\label{t_menos}
	||t^-(w_k)w_k||^2_A= \int_{\R^N} a_-| t^-(w_k)w_k|^qdx +\int_{\R^N} b_-| t^-(w_k)w_k|^p dx.
	\end{equation}
	As $w_0$ is a ($E_{\infty}$) solution and remembering that the functional associated with ($E_{\infty}$) is given by $I(u)=\frac{1}{2}||u||_A^2-\frac{1}{p}||u||_p^p,$ and $I'(u)=||u||_A^2- ||u||_p^p$ we have
	$$I'(w_0)w_0=||w_0||_A^2- ||w_0||_p^p=0 \Rightarrow ||w_0||_A^2= ||w_0||_p^p,$$
	whence
	\begin{eqnarray*}
		\nonumber  m_{\infty} &=& I(w_0)=\frac{1}{2}||w_0||_A^2-\frac{1}{p}||w_0||_p^p \\
		\nonumber    &=&  \frac{1}{2}||w_0||_A^2-\frac{1}{p}||w_0||_A^2=\frac{p-2}{2p}||w_0||_A^.
	\end{eqnarray*}
	Being $w_0$ solution of ($E_{\infty}$) it follows that $w_k(x)=w_0(x+ke)$ so it is, with this and for  $ I'(w_0)w_0=0,$ we have $I'(w_k)w_k=0$. So
	\begin{equation}\label{m_infty>0}
	|| w_k||^2_A= \int_{\R^N}  |  w_k|^qdx =\frac{2p}{p-2}m^{\infty} \;\; \mbox{for all}\;\;k\geq 0.
	\end{equation}
	
	%	Ainda, pela condição (A) e por um argumento semelhante ao utilizado na demonstração do Lemma \ref{Lemmadecompacidade} we have
	It is known that $ w_n $ is bounded in $L^{r'}$ and $ w_n \rightarrow 0 $ a.e., by Theorem \cite[Theorem 13.44]{Hewitt} that $ w_n \rightharpoonup 0 $ weakly in $ L^{r'}$. From condition ($A$), $ a_- \in (L^{r'})'=L^r$ we get
	\begin{equation}\label{int_a}
	\int_{\R^N} a_- |w_k|^qdx \rightarrow 0 \;\; \mbox{as}\;\;k\rightarrow \infty.
	\end{equation}
	In addition, by $(B_1)$ and $(B_2)$
	\begin{equation}\label{int_b}
	\int_{\R^N}(1-b_1) | w_k|^qdx =  \int_{B(0,R)}(1-b_1) | w_k|^qdx+ \int_{B^c(0,R)}(1-b_1) | w_k|^qdx \rightarrow 0 \;\; \mbox{as}\;\;w_k\rightarrow \infty.
	\end{equation}
	By (\ref{t_menos}), (\ref{int_a}) and (\ref{int_b}) we have that $t^-(w_k)\rightarrow 1 $ as $ k\rightarrow \infty.$ Likewise
	$$ \lim_{k\rightarrow \infty} J_{a_0,b_0} (t^-(w_k)w_k)= \lim_{k\rightarrow \infty} J_{\infty} (t^-(w_k)w_k)=m_{\infty}.$$
	Thus
	\begin{eqnarray}\label{geq}
	m_{\infty} &=& \inf_{u \in M^{\infty}}J_{\infty}(u)=\lim_{k\rightarrow \infty} J_{\infty} (t^-(w_k)w_k)\geq \inf_{u \in M^-_{a_0,b_0}}J_{a_0,b_0}(u).
	\end{eqnarray}
	We also have to $u\in  M_{a_0,b_0}$, by Lemma \ref{Nehari}(i), $J_{a_0,b_0}(u)=\sup_{ t\geq 0}J_{a_0,b_0}(tu),$ and more, there is a single $t^{\infty}>0$ such that $t^{\infty}u \in M^{\infty}$. So
	\begin{eqnarray}
	\nonumber  J_{a_0,b_0} (t^{\infty}u)   &=&\frac{1}{2}||t^{\infty}u||_A^2-\frac{(t^{\infty})^q}{q}\int_{\R^N}a_-|u|^qdx  -\frac{(t^{\infty})^p}{p}\int_{\R^N}b_1|u|^pdx  \\
	\nonumber    &\geq&  \frac{1}{2}||t^{\infty}u||_A^2  -\frac{(t^{\infty})^p}{p}\int_{\R^N}|u|^pdx \\
	\nonumber   &=&  J_{\infty} (t^{\infty}u) \geq m_{\infty},
	\end{eqnarray}
	whence 
	\begin{equation}\label{leq}
	\inf_{u \in M_{a_0,b_0}} J_{a_0,b_0} (t^{\infty}u)\geq m_{\infty}.
	\end{equation}
	In this way, by (\ref{geq}) and (\ref{leq})
	$$\inf_{u \in M_{a_0,b_0}}J_{a_0,b_0}(u)=\inf_{u \in M^{\infty}}J_{\infty}(u)=m^{\infty}.$$
\end{proof}

The following result shows that the infimum is not assumed in $M_{a_0,b_0}.$

\begin{lemma}
	Problem $(P_{a_0,b_0})$ admits no solution $u_0$ such that $J_{a_0,b_0}(u_0)=\inf_{u \in M_{a_0,b_0}}J_{a_0,b_0}(u)$.
\end{lemma}

\begin{proof}
	Suppose by contradiction that exists $u_0$ such that $J_{a_0,b_0}(u_0)=\inf_{u \in M_{a_0,b_0}}J_{a_0,b_0}(u)$. Thus, by Lemma \ref{Nehari}(i), $J_{a_0,b_0}(u_0)=\sup_{ t\geq 0}J_{a_0,b_0}(tu_0),$ and more, there is a single $t_{u_0}>0$ such that $t_{u_0} u_0 \in M^{\infty}$. Thus, by $ (A) $ and by the previous Lemma
	\begin{eqnarray}
	\nonumber    m^{\infty} &=&  \inf_{u \in M^{\infty}}J_{\infty}(u)=\inf_{u \in M^-_{a_0,b_0}}J_{a_0,b_0}(u)=J_{a_0,b_0}(u_0) \\
	\nonumber    &\geq&  J_{a_0,b_0}(t_{u_0}u_0)\geq J_{\infty}(t_{u_0}u_0) - \frac{t_{u_0}^q}{q}\int_{\R^N}a_-|u_0|^qdx\\
	\nonumber    &\geq &m^{\infty}  - \frac{t_{u_0}^q}{q}\int_{\R^N}a_-|u_0|^qdx.
	\end{eqnarray}
	This implies that $\int_{\R^N}a_-|u|^qdx=0$ and still $u_0\equiv 0$ where $a_-(x)\neq 0 $. For the hypothesis $ (A) $ the set of points of $ \R^N $ where $ a_- (x)> 0 $ has positive measure, which implies that the set where $ u_0 \equiv 0 $ has positive measure.
	
	We also have to $  m^{\infty} = \inf_{u \in M^{\infty}}J_{\infty}(u)= J_{\infty}(t_{u_0}u_0),$ that is, $t_{u_0}u_0$ is a $(E_{\infty})$ solution, and more, using the same argument used in the proof of Theorem \ref{regul} (ii) we obtain that  $|t_{u_0}u_0|$ is positive solution of $(E_{\infty})$. This contradicts $ u_0 \equiv 0 $ in a positive measure set, completing the proof.
\end{proof}

\begin{lemma}\label{PS.seq.J.infty}
	Let $\{u_n\}$ be a minimizing sequence in $M_{a_0,b_0}$ for the functional $J_{a_0,b_0}$. Then,
	\begin{description}
		\item[  $(i)$ ]  $ \int_{\R^N}a_-|u_n|^qdx=o(1);$
		\item[$(ii)$   ]   $\int_{\R^N}(1-b_1)|u|^pdx=o(1)$.
	\end{description}	
	Moreover, $\{u_n\}$ is a $(PS)_{m_{\infty}}$ sequence in $\HA$ for $J_{\infty}$.
\end{lemma}

\begin{proof}
	For each $n \in \N$ there is only one $t_n>0$ such that $t_nu_n \in M^{\infty}$, that is
	$$
	J'_{\infty}(t_nu_n )t_nu_n =t_n^2||u_n||_A^2-t_n^p\int_{\R^N}|u_n|^pdx=0\;\; \Rightarrow  \;\; t_n^2||u_n||_A^2=t_n^p\int_{\R^N}|u_n|^pdx.
	$$
	Now, by Lemma \ref{Nehari}(i)
	\begin{eqnarray}
	\nonumber    \sup_{t\geq 0} J_{a_0,b_0}(t u_n)&= &J_{a_0,b_0}( u_n) \geq   J_{a_0,b_0}(t_n u_n)\\
	\nonumber    &=&  J_{\infty} (t_nu_n ) + \frac{t_n^p }{p} \int_{\R^N}(1-b_1)|u_n|^pdx - \frac{t_n^q }{q} \int_{\R^N}a_-|u_n|^qdx\\
	\nonumber    &\geq&m_{\infty} + \frac{t_n^p }{p} \int_{\R^N}(1-b_1)|u_n|^pdx - \frac{t_n^q }{q} \int_{\R^N}a_-|u_n|^qdx.
	\end{eqnarray}
	Since $ \{u_n \} $ is a minimizing sequence in $M_{a_0,b_0}$ for the functional $J_{a_0,b_0}$ and by Lemma \ref{5.1.a} we have
	$$ J_{a_0,b_0}(u_n)=  \inf_{u \in M^-_{a_0,b_0}} J_{a_0,b_0} ( u)+o(1)= m^-_{a_0,b_0}+o(1)=m^{\infty}+o(1). $$
	So
	$$\frac{t_n^p }{p} \int_{\R^N}(1-b_1)|u_n|^pdx=o(1)$$
	and
	$$\frac{t_n^q }{q}\int_{\R^N}a_-|u_n|^qdx=o(1).$$
	
	To ensure the result from these two equalities above, we need to show that $ t_n \nrightarrow 0 $, that is, that exists $c_0>0$ such that $t_n>c_0$ for all n. 
	Suppose by contradiction that  $t_n \rightarrow 0$ as $n\rightarrow \infty $. As $J_{a_0,b_0} ( u_n)=m^{\infty}+o(1), $ by Lemma \ref{bounded}, $||u_n||_A$ is uniformly bounded and therefore $||t_n u_n||_A \rightarrow 0$ as $n\rightarrow \infty $.
	
	As $t_nu_n\in M^{\infty}$, follow that
	$$J_{\infty}(t_nu_n)= \left( \frac{p-2}{2p} \right)||t_nu_n ||_A^2 \rightarrow 0, $$
	which contradicts the fact $J_{\infty}(t_nu_n)\geq  m^{\infty}>0$. Thus
	$$ \int_{\R^N}(1-b_1)|u_n|^pdx=0(1)$$
	and
	$$ \int_{\R^N}a_-|u_n|^qdx=0(1).$$
	This implies that $||u_n ||_A^2 = \int_{\R^N}|u_n|^pdx+o(1) $ and so $J_{\infty}( u_n)=  m^{\infty}+o(1)$.
	
	We can conclude the last statement of Lemma using the  \cite[Lemma7]{wang.wu}, % ou \ref{Lemma7wang.wu}, 
	whence we get that  $\{u_n\}$ is a sequence $(PS)_{m^{\infty}}$ in $\HA$ for $J_{\infty}.$

\end{proof}

The following result is of fundamental importance in obtaining the third solution.

\begin{lemma}\label{Lemma5.3}
	Consider the set $  \mathcal{C}= \{  u\in M^-_{a_0,b_0}; J_{a_0,b_0}(u)<m_{\infty} +l_0\}$. There is a $ l_0> 0 $ such that 
	$$ \int\frac{x}{|x|} (|\nabla u |^2 +u^2) \neq 0  $$
	for all $u\in \mathcal{C}$.
\end{lemma}

\begin{proof}
	Suppose by contradiction that there is no such $l_0  $, then, there will be a sequence $\{l_n\}$ such that for all $n\in \N$ exist $u  \in M^-_{a_0,b_0}$ with $J_{a_0,b_0}(u)\leq m_{\infty} +l_n,$ but $ \int\frac{x}{|x|} (|\nabla u |^2 +u^2) = 0  $.So we can take  $\{u_n\} \in M^-_{a_0,b_0}$ so that $J_{a_0,b_0}(u)= m_{\infty} +o(1)\leq m_{\infty} +l_0.$ With this and by Lemma \ref{PS.seq.J.infty} follows that $\{u_n\}$ is a $(PS)_{m_{\infty}} $-sequence in $\HA$ for $J_{\infty}$.
	Using Lemma \ref{bounded} there will be a subsequence $\{u_n\}$ and $u_0 \in \HA$ such that $ u_n \rightharpoonup u_0$ weak in $\HA$. By Splitting Lemma as in \cite[Lemma2.3]{furt.maia}, there will be a sequence $\{x_n\} \subset \R^N,$ and a positive solution $w_0 \in \HA$ of $(E^{\infty})$ such that
	\begin{equation}\label{conv.seq}
	||u_n(x) - w_0(x-x_n)||_{A} \rightarrow 0 \;\;\; \mbox{as} \;\; n\rightarrow \infty.
	\end{equation}
	
	We will now show that $|x_n| \rightarrow \infty$ as $n \rightarrow \infty$.In fact, suppose otherwise, then we will have a bounded sequence $ \{x_n \} $ and there will be $x_0 \in \R^N$ such that $x_n \rightarrow x_0$. Hence, by (\ref{conv.seq})
	\begin{eqnarray}
	\nonumber  \int_{\R^N}  a_-|u_n|^qdx &=& \int_{\R^N}  a_-(x)|w_0( x-x_n)|^qdx +o(1) \\
	\nonumber    &=&   \int_{\R^N}  a_-(x+x_n)|w_0( x )|^qdx +o(1) \\
	\nonumber    &=& \int_{\R^N}  a_-(x+x_0)|w_0( x )|^qdx +o(1),
	\end{eqnarray}
	which is absurd for the result obtained in Lemma \ref{PS.seq.J.infty}. So we can assume that $\frac{x_n}{|x_n|}\rightarrow e$ as $n\rightarrow \infty ,$ where $e \in S^{N-1}$. By Lebesgue's Dominated Convergence Theorem, we have
	\begin{eqnarray}
	\nonumber  0&=& \int_{\R^N}\frac{x}{|x|} (|\nabla u_n|^2 +u_n^2 )dx =  \int_{\R^N}\frac{x +x_n}{|x+x_n|} (|\nabla w_n|^2 +w_n^2 )dx\\
	\nonumber    &=&  \frac{2p}{p-2}m_{\infty}e+o(1).
	\end{eqnarray}
	Coming into a contradiction. This brings us to the result we wanted to demonstrate.
\end{proof}

In the next Lemmas we will establish some necessary estimates to arrive at the last result of this chapter. For this, we will make the following considerations. 
By (\ref{2.3.1}), (\ref{2.6}) and Lemma \ref{Nehari}, for each $u \in M^-_{a_{\lambda},b_{\mu}}$ exist only one $t_0^-(u)>0$ such that $t_0^-(u)u \in M_{a_0,b_0} $ and
$$ t_0^-(u)>  t_{\max}(u)=\left(\frac{(2-q)||u||^2_A}{(p-q)\intg}\right)^{\frac{1}{p-2}}      >0. $$
%Remember that $$\Upsilon_0=( 2-q )^{ 2-q } \left(  \frac{p-2}{ ||a_+||_{q'} }\right)^{p-2}   \left(  \frac{S_p }{p-q }\right)^{p-q}. $$
In addition, consider
$$ \theta_{\mu}=\left[ \frac{(p-q)(1+\mu ||b/a||_{\infty})}{2-q}\left( 1+||a_-||_{q^*}\left( \frac{(p-q)(1+\mu ||b/a||_{\infty})}{(2-q)S_p^{\frac{p-q}{2-q}}} \right)^{\frac{2-q}{p-2}} \right) \right]^{\frac{p}{p-2}}.$$
Thereby we present the following results.

\begin{lemma}\label{lem5.4}
	For each $\lambda >0$ and $\mu>0$ with $ \lambda^{p-2}( 1+\mu|| b_2||_{\infty})^{2-q}<(q/2)^{p-2}\Upsilon_0$, we have
	
	$(i)\; [t_0^-(u)]^p< \theta_{\mu} $ for all $ u \in M^-_{a_{\lambda},b_{\mu}}$;
	
	$(ii) \; \int_{\R^N} b_0|u|^p dx \geq \frac{pq}{\theta_{\mu}(p-q)} m^{\infty} $ for all $u\in M^-_{a_{\lambda},b_{\mu}}.$
\end{lemma}
\begin{proof}
	$(i)$ For $u\in M_{a_{\lambda},b_{\mu}}$ we have
	$$ ||u||^2_A - \int_{\R^N} a_{\lambda} |u|^qdx - \int_{\R^N}b_{\mu}|u|^pdx=0.$$
	In particular, as $ u \in M^-_{a_{\lambda},b_{\mu}}$ follow that $F''_u(1)<0$ whence
	$$ (2-q)||u||^2_A < (p-q)\int_{\R^N}b_{\mu}|u|^pdx.  $$
	
	We will divide from here into two cases: first, if $t_0^-(u)<1. $ As $\theta_{\mu}>1$ for all $\mu>0$ we have $t_0^-(u)<1<\theta_{\mu},$ as wished. In another case, if $t_0^-(u)\geq 1, $
	\begin{eqnarray}
	\nonumber [t_0^-(u)]^p  \int_{\R^N}b_0|u|^pdx  &=& [t_0^-(u)]^2 ||u||^2_A-[t_0^-(u)]^q \int_{\R^N} a_- |u|^qdx \\
	\nonumber    &=&  [t_0^-(u)]^2\left(  ||u||^2_A-[t_0^-(u)]^{q-2} \int_{\R^N} a_- |u|^qdx\right)\\
	\nonumber    &\leq&  [t_0^-(u)]^2\left(  ||u||^2_A- \int_{\R^N} a_- |u|^qdx\right)
	\end{eqnarray}
	then
	\begin{equation}
	[t_0^-(u)]^{p-2}\leq\frac{ ||u||^2_A- \int_{\R^N} a_- |u|^qdx}{\int_{\R^N}b_0|u|^pdx}.
	\end{equation}
	In addition, by (\ref{2.3.1}) and Sobolev's inequality,
	\begin{eqnarray}\label{5.3}
	\nonumber    ||u||^2_A &<& \frac{p-q}{2-q} \int_{\R^N}b_{\mu}|u|^pdx \leq \frac{p-q}{2-q} (1+\mu||b_2/b_1||_{\infty})\int_{\R^N}b_0|u|^pdx  \\
	\nonumber    &=&\frac{(p-q)(1+\mu||b_2/b_1||_{\infty}) }{2-q  } ||u||^p_p\\
	&=&\frac{(p-q)(1+\mu||b_2/b_1||_{\infty}) }{(2-q)S_p^{p/2}} ||u||^p_A.
	\end{eqnarray}
	Then
	\begin{eqnarray}\label{5.4}
	\nonumber  ||u||^{p-2}_A   &\geq&\left( \frac{(2-q)S_p^{p/2} }{ (p-q)(1+\mu||b_2/b_1||_{\infty})   }\right)^{\frac{1}{p-2}}.
	\end{eqnarray}
	The consequences of these three inequalities presented above
	\begin{eqnarray}
	\nonumber   [t_0^-(u)]^{p-2}  &\leq&\frac{ ||u||^2_A- \int_{\R^N} a_- |u|^qdx}{\int_{\R^N}b_0|u|^pdx}=\frac{ 1}{\int_{\R^N}b_0|u|^pdx}||u||^2_A \left( 1+ \frac{ \int_{\R^N} a_- |u|^qdx }{||u||^{2}_A} \right) \\
	%\nonumber    &\leq&   (1+\mu||b_2/b_1||_{\infty}) \left(  \frac{p-q}{2-q}  \right)  \left(1+\frac{ \int_{\R^N} a_- |u|^qdx }{||u||^{2}_A}\right) \\
	%\nonumber    &\leq&  (1+\mu||b_2/b_1||_{\infty}) \left(  \frac{p-q}{2-q}  \right)  \left(1+\frac{||a_-||_{q^*}}{||u||^{2-q}_AS_p^{q/2}}\right) \\
	\nonumber    &\leq& (1+\mu||b_2/b_1||_{\infty}) \left(  \frac{p-q}{2-q}  \right)  \left(1+||a_-||_{q^*}\left(\frac{(p-q)(1+\mu||b_2/b_1||_{\infty}) }{(2-q)S_p^{\frac{p-q}{2-q}}}\right)^{\frac{2-q}{p-2}}\right)=\theta_{\mu}^{\frac{p-2}{p}}.
	\end{eqnarray}
	
	In conclusion, $[t_0^-(u)]^{p }\leq \theta_{\mu}$, for all $u \; \in \; M^-_{a_{\lambda},b_{\mu}}$.

	$(ii)$ We will use in this part the result shown in Lemma \ref{5.1.a} which says that $\inf_{u \in M^-_{a_0,b_0}}J_{a_0,b_0}(u) =m^{\infty}.$ For $t_0^-(u)u \in M^-_{a_0,b_0},$
	\begin{eqnarray}
	\nonumber   m^{\infty}  &=&  J_{a_0,b_0}(t_0^-(u)u)\\
	\nonumber    &=&  \left( \frac{1}{2}- \frac{1}{q}\right)[t_0^-(u)]^2||u||^{2}_A +\left( \frac{1}{q}- \frac{1}{p}\right)[t_0^-(u)]^p\int_{\R^N}b_-|u|^pdx\\
	\nonumber    &<& \left( \frac{1}{q}- \frac{1}{p}\right)[t_0^-(u)]^p \int_{\R^N}b_-|u|^p dx,
	\end{eqnarray}
	whence we get
	$$
	\int_{\R^N}b_-|u|^pdx\geq \frac{1}{[t_0^-(u)]^p}\left( \frac{pq}{p-q}\right)m^{\infty}.
	$$
	Knowing that $[t_0^-(u)]^{p-2}\leq \theta_{\mu}^{\frac{p-2}{p}} $, we can conclude that for $u \in M^-_{a_0,b_0}$
	$$
	\int_{\R^N}b_-|u|^pdx\geq \frac{pq}{\theta_{\mu}(p-q)}m^{\infty}.
	$$
	
\end{proof}

In the next chapter we will use a category theory result to obtain a third solution. For this we will need to construct a certain homotopy and the following result in the mune of tools for this construction. 
To assert that the set of functions $ u \in \HA $ satisfying Lemma conditions is not an empty set, we recall that at the end of Proposition \ref{prop4.1} we show that $ k_2 >0$ such that for all $k>k_2,$ there will be a $t^*>0$ such that $u^+_{\lambda,\mu}+t^*w_k \in  M^-_{a_{\lambda},b_{\mu}} $ with $ \jf(u^+_{\lambda,\mu}+t^*w_k)< m^+_{a_{\lambda},b_{\mu}}+m^{\infty}$.

\begin{lemma}\label{Lemma5.5}
	There exists $\lambda_0>0$ and $\mu_0>0$ with
	$$
	\lambda_0^{p-2}(1+\mu_0||b_1||_{\infty})^{2-q}<\left( \frac{q}{2}\right)^{p-2}\Upsilon_0,
	$$
	such that for all $\lambda \in (0,\lambda_0)$ and $\mu \in (0,\mu_0)$, we have
	$$\int_{\R^N}\frac{x}{|x|}(|\nabla u|^2 + u^2)dx \neq 0$$
	for all $u\in M^-_{a_{\lambda},b_{\mu}}$ with $ \jf(u)< m^+_{a_{\lambda},b_{\mu}}+m^{\infty}$.
\end{lemma}

\begin{proof}
	Let $u\in  M^-_{a_{\lambda},b_{\mu}} $ be with $ \jf(u)< m^+_{a_{\lambda},b_{\mu}}$. By Lemma \ref{Nehari}(i) there will be a $ t_0^-(u)>0$ and also using \ref{teorema3.1}(ii) we will have
	\begin{eqnarray}
	\nonumber   \jf(u)  &=& \sup_{t\geq 0}\jf(tu)\geq \jf(t_0^-(u)u)   \\
	\nonumber    &=&   J_{a_{0},b_{0}}(t_0^-(u)u)-\frac{\lambda[t_0^-(u) ]^q}{q} \int_{\R^N}a_+|u|^q dx \\
	\nonumber    &  -&   \frac{\mu[t_0^-(u) ]^p}{p} \int_{\R^N}b_2|u|^p dx  .
	\end{eqnarray}
	Now, using Lemma \ref{lem5.4} and the inequalities of H\"{o}lder and Sobolev
	\begin{eqnarray}
	\nonumber   J_{a_{0},b_{0}}(t_0^-(u)u)  &\leq& \jf(u)  +\frac{\lambda[t_0^-(u) ]^q}{q} \int_{\R^N}a_+|u|^q dx  + \frac{\mu[t_0^-(u) ]^p}{p} \int_{\R^N}b_2|u|^p dx  \\
	\nonumber    &< &  m^+_{a_{\lambda},b_{\mu}}+m^{\infty}  +\frac{\lambda\theta_{\mu}^{q/p}}{q} ||a_+||_{q^*}S_p^{-\frac{q}{2}}||u||^q_A  + \frac{\mu\theta_{\mu}||b_2||_{\infty}}{p} S_p^{-\frac{p}{2}}||u||^p_A. 
	\end{eqnarray}
	See that by hypothesis $ \jf(u)< m^+_{a_{\lambda},b_{\mu}}+m^{\infty} < m^{\infty}$. With this and by Lemma (\ref{bounded}), for each $\lambda>0$ and $ \mu>0$ with $\lambda^{p-2}(1+\mu||b_2||_{\infty})^{2-q}<(\frac{q}{2})\Upsilon_0,$ there will be a constant $\tilde{c}$ independent of $\lambda $ and $\mu$ such that $||u||_A\leq  \tilde{c}   $ for all $u \in M^-_{a_{\lambda},b_{\mu}}$ with $ \jf(u)< m^+_{a_{\lambda},b_{\mu}}+m^{\infty} $. Then,
	$$
	J_{a_{0},b_{0}}(t_0^-(u)u) \leq m^+_{a_{\lambda},b_{\mu}}+m^{\infty} +\frac{\lambda\theta_{\mu}^{q/p}}{q} ||a_+||_{q^*}S_p^{-\frac{q}{2}}\tilde{c}^q  + \frac{\mu\theta_{\mu}||b_2||_{\infty}}{p} S_p^{-\frac{p}{2}}\tilde{c}^p.
	$$
	
	Let $l_0>0$ be as in Lemma \ref{Lemma5.3}. Then there will be $\lambda_0$ and $\mu_0$ positive with $\lambda^{p-2}(1+\mu||b_2||_{\infty})^{2-q}<(\frac{q}{2})\Upsilon_0,$ such that for $\lambda \in (0,\lambda_0)$ and $\mu \in (0,\mu_0)$
	\begin{equation}\label{5.5}
	J_{a_{0},b_{0}}(t^-(u)u)<m^{\infty} +l_0.
	\end{equation}
	As $t_0^-(u)u \in M_{ a_{0},b_{0} }$ and $t_0^-(u)>0,$ by Lemma \ref{Lemma5.3} and \ref{5.5}
	$$
	\int_{\R^N} \frac{x}{|x|}( |\nabla t_0^-(u)u|^2+(t_0^-(u)u)^2)dx \neq 0 ,
	$$
	giving us that
	$$
	\int_{\R^N} \frac{x}{|x|}( |\nabla u|^2+u)^2)dx \neq 0 ,
	$$
	for all $u \in M^-_{a_{\lambda},b_{\mu}}$ with $ \jf(u)< m^+_{a_{\lambda},b_{\mu}}+m^{\infty} $.
\end{proof}

\subsection{Getting a Third Solution}

To show the Theorem \ref{teo1.1(ii)} we will present some concepts necessary to apply the category theory. This idea was used by Adach and Tanaka \cite{AT}.

\begin{definition}
	Let $ X $ be a topological space. A non-empty subset $ Y \subset X $ is said to be contractile if it exists $x_0 \in X$ and continuous map $\gamma:[0,1]\times Y \rightarrow X$, such that 
	$$\gamma (0,x)=x\;\;;\ \mbox{for all}\;\;x \; \in \; Y,$$
	and
	$$\gamma (1,x)=x_0\;\;;\ \mbox{pfor all}\;\;x \; \in \; Y.$$
	
\end{definition}

\begin{definition}
	We define
	\begin{eqnarray}
	\nonumber cat(X)=\min    &\{&  m\in \N ; \exists m \mbox{ closed subsets } Y_1, ..., Y_m \subset X, \mbox{ such that } \\
	\nonumber    &&   Y_i \mbox{ is contractile to a point of }X \mbox{ for all } \cup^{m}_{i=1}Y_i =X \}.
	\end{eqnarray}
	In case there is no finite coverage for $ X $ of sets $ Y_1, ..., Y_n \subset X $, such that $ Y_i $ is contractile to a point of $ X $ for all $ i \in \N $ we say that $cat(X)=\infty.$
\end{definition}

In order to obtain our result, we will also need the two results that we will enunciate next.

\begin{lemma}\label{Lemma6.2}
	Let $ X $ is a Hilbert manifold and $J \in C^1(X,\R)$. Assume there is $c_0 \in \R$ and $k\in \N,$\\
	$(i)\; J$ satisfies the PS condition for all level $c<c_0;$\\
	$(ii)\; cat(\{ x\in X; F(x) \leq c_0 \}) \geq k.$
	
	Then $ J (x) $ has at least $ k $ critical points in $\{ x \in X; J(x) \leq c_0 \}.$
\end{lemma}
\begin{proof}
	According Ambrosetti \cite[Theorem 2.3]{Amb}
\end{proof}

\begin{lemma}\label{Lemma6.3}
	Let $ X $ be a topological space. Suppose there are continuous maps
	$$\Phi : S^{N-1} \rightarrow X,\;\;\; \Psi : X\rightarrow S^{N-1},$$
	
	\noindent such that $\Psi \circ \Phi$ is homotopic to the identity in $ S^{N-1}$, that is, there is a continuous application 
	$\xi:[0,1] \times S^{N-1} \rightarrow S^{N-1}$, such that
	$$\xi (0,x)= \Psi \circ \Phi( x)\;\;;\ \mbox{for each}\;\;x \; \in \; S^{N-1} ,$$
	and
	$$\xi (1,x)=x \;\;;\ \mbox{for each}\;\;x \; \in \; S^{N-1}.$$
	
	Then $cat(X)\geq 2.$
\end{lemma}
\begin{proof}
	According Adachi and Tanaka \cite[Lemma 2.5]{AT}
\end{proof}

To use this Lemma that we have just enunciated, we will try to show that for a $ \varepsilon> 0 $ small enough
\begin{equation}\label{6.1}
cat([ \jf\leq m^+_{a_{\lambda},b_{\mu}}+m^{\infty} -\varepsilon])\geq 2.
\end{equation}
Consider $k_2$ and $u^+ + s_lt_0w_k$ according to Proposition \ref{prop4.1}. For $k>k_2,$ we define the map $ \Phi_{a_{\lambda},b_{\mu}} : S^{N-1} \rightarrow \HA$ by
$$  \Phi_{a_{\lambda},b_{\mu}}(e)(x)=u^+ + s_lt_0w_0(x+ke)=  u^+ + s_lt_0w_k \;\; \mbox{for} \;\; e\in S^{N-1}.$$
For $c\in \R^+,$ we denote
$$[\jf\leq c] =\{ u\in M^-_{a_{\lambda},b_{\mu}}; u\geq 0, \jf\leq c\}.$$
We then have the following result

\begin{lemma}\label{Lemma6.4}
	There is a $\{\varepsilon_n\} \subset \R^+$ with $ \varepsilon_n\rightarrow \infty$ as $n\rightarrow \infty$ such that
	$$\Phi_{a_{\lambda},b_{\mu}}(S^{N-1})\subset [ \jf\leq m^+_{a_{\lambda},b_{\mu}}+m^{\infty }-\varepsilon_n]. $$
\end{lemma}

\begin{proof}
	According to Proposition \ref{prop4.1}, for each $k>k_2,$ we have $u^+ + s_lt_0w_k \in M^-_{a_{\lambda},b_{\mu}}$ and
	$$  \sup_{ t\geq 0}  \jf( u^+ +tw_k )<m^+_{a_{\lambda},b_{\mu}}+m^{\infty} $$
	uniformly $e\in S^{N-1};$ remembering that $ w_k(x)=w_0(x+ke)$.
	Note that $\Phi_{a_{\lambda},b_{\mu}}(S^{N-1})$ it is a compact set. From there we obtain that $\jf(u^+ + s_lt_0w_k)\leq m^+_{a_{\lambda},b_{\mu}}+m^{\infty }-\varepsilon_n.$ Whence we conclude what we wanted to demonstrate.

\end{proof}

Let us denote $\mathcal{Q}_{\lambda,\mu}=\{ u\in M^-_{a_{\lambda},b_{\mu}}; \jf<  m^+_{a_{\lambda},b_{\mu}}+m^{\infty } \}$. Our goal is to show that this subset of $M^-_{a_{\lambda},b_{\mu}}$ has a category greater than or equal to $2$. For this, let us start by defining the following function $ \Psi_{a_{\lambda},b_{\mu} }$, according to Lemma \ref{Lemma5.5}

$$ \Psi_{a_{\lambda},b_{\mu} }:\mathcal{Q}_{\lambda,\mu} \rightarrow S^{N-1}  $$
by
$$ \Psi_{a_{\lambda},b_{\mu} }( u)=\frac{  \int_{\R^N} \frac{x}{|x|}( |\nabla_A u|^2+u)^2)dx}{| \int_{\R^N} \frac{x}{|x|}( |\nabla_A u|^2+u)^2)dx|}.$$
From these considerations we will show the following result, in which we construct a homotopic application to identity, necessary to apply the category theory to our problem.

\begin{lemma}\label{Lemma6.5}
	Let $\lambda_0$ and $\mu_0$ be as in Lemma \ref{Lemma5.5}. Then, for each $\lambda  \in (0,\lambda_0)$ and $\mu  \in (0,\mu_0)$, there will be $k_*\geq k_2$ such that for $k>k_*$, the map
	$$ \Psi_{a_{\lambda},b_{\mu} } \circ \Phi_{a_{\lambda},b_{\mu} }:S^{N-1} \rightarrow  S^{N-1}$$
	is homotopic to identity.
\end{lemma}

\begin{proof}
	Consider $ \sum = \{ u \in \HA\setminus \{0\};   \int_{\R^N} \frac{x}{|x|}( |\nabla_A u|^2+u)^2)dx \neq 0  \}$.
	We define
	$$  \overline{\Psi}_{a_{\lambda},b_{\mu} }: \sum \rightarrow S^{N-1}$$
	by
	$$   \overline{\Psi}_{a_{\lambda},b_{\mu} }(u) = \frac{  \int_{\R^N} \frac{x}{|x|}( |\nabla_A u|^2+u)^2)dx}{| \int_{\R^N} \frac{x}{|x|}( |\nabla_A u|^2+u)^2)dx|} $$
	an extension of $ \Psi_{a_{\lambda},b_{\mu} }$.
	
	See that $ \Psi $ is indeed an extension since under the assumptions of Lemma \ref{Lemma5.5}, for all $u\in M^-_{a_{\lambda},b_{\mu}}$ com $\jf(u)< m^+_{a_{\lambda},b_{\mu}}+m^{\infty },$ we have $\int_{\R^N} \frac{x}{|x|}( |\nabla_A u|^2+u)^2)dx \neq 0 $, whence follows the inclusion
	$$   [ \jf< m^+_{a_{\lambda},b_{\mu}}+m^{\infty }] \subset \sum = \{ u \in \HA\setminus \{0\};   \int_{\R^N} \frac{x}{|x|}( |\nabla_A u|^2+u)^2)dx \neq 0  \}.$$
	As $w_k \in \sum$ for all $e \in S^{N-1}$ and $k$ sufficiently large, we take $ \gamma:[s_1,s_2] \rightarrow S^{N-1},$, a regular geodesic between $ \overline{\Psi}_{a_{\lambda},b_{\mu} }(w_k)$ and $ \overline{\Psi}_{a_{\lambda},b_{\mu} }(\Phi_{a_{\lambda},b_{\mu} }(e))$ where
	$$ \gamma(s_1)=  \overline{\Psi}_{a_{\lambda},b_{\mu} }(w_k)$$
	and
	$$ \gamma(s_2)=  \overline{\Psi}_{a_{\lambda},b_{\mu} }(\Phi_{a_{\lambda},b_{\mu} }(e)).$$
	
	Remember that $w_k(x)=w_0(x+ke) $ for $k \in \R$, $e \in S^{n-1}$ and $ \Phi_{a_{\lambda},b_{\mu} }(e)(x)= u^+_{\lambda,\mu}+s_lt_0 w_k(x);$ with $s_l$ as in the proof of Proposition \ref{prop4.1} such that $u^+_{\lambda,\mu}+s_lt_0 w_k(x) \in M^-_{a_{\lambda},b_{\mu}}. $
	By an argument similar to that used in Lemma \ref{Lemma5.3}, there is a $k_*\geq k_2$ such that for $k>k_*,$
	
	$$ w_0 \left( x+ \frac{k}{2(1-\theta)}e \right) \in \sum  \;\; \mbox{  for all  }  \;\; e \in S^{N-1} \;\; \mbox{ and } \;\; \theta \in [1/2,1),$$
	where $ w_0 $ is a positive solution of $P_{\infty}$ in $\HA.$
	
	We now want to construct a $ \zeta $ function so that its composition with $ \overline {\Psi}$ is homotopic to identity. For this we define
	$$ \zeta_k(\theta,e):[0,1]\times S^{N-1} \rightarrow S^{N-1} $$
	by
	$$\zeta_k(\theta,e)=
	\left\{ \begin{array} [c]{ll}
	\gamma(2\theta(s_1-s_2)+s_2) \, \, &\mbox{for} \, \,\theta \in  \lbrack 0, 1/2) ;\\
	\overline{\Psi}_{a_{\lambda},b_{\mu} }(w_0(x+\frac{k}{2(1-\theta)}e)) \, \, &\mbox{for} \, \,\theta \in   \lbrack  1/2,1) ;\\
	e \, \, &\mbox{for} \, \,\theta =1.
	
	\end{array}
	\right.
	$$
	Then
	\begin{eqnarray}
	\nonumber \zeta_k(0,e)    &=&\gamma(2^.0(s_1-s_2)+s_2)=\gamma(s_2) \\
	\nonumber    &=&  \overline{\Psi}_{a_{\lambda},b_{\mu} }( \Phi_{a_{\lambda},b_{\mu} }(e) )=\overline{\Psi}_{a_{\lambda},b_{\mu} }( \Phi_{a_{\lambda},b_{\mu} }(e) )
	\end{eqnarray}
	which makes sense since $J_{a_{\lambda},b_{\mu}} (\Phi(e))=J_{a_{\lambda},b_{\mu}} (u^+_{\lambda,\mu}+s_lt_0\Psi_k )<m^+_{a_{\lambda},b_{\mu}}+m^{\infty }.$ Besides that,
	$$ \zeta_k(1,e)=e.$$
	It has already been seen that  $ u^+_{\lambda,\mu} \in C(\C )$. We must also show that $\lim_{\theta \rightarrow 1^-} \zeta_k(\theta,e)=e$ and $\lim_{\theta \rightarrow \frac{1}{2}^-} \zeta_k(\theta,e)=   \overline{\Psi}_{a_{\lambda},b_{\mu} }(w_0(x+ke)) . $
	Firstly, note that $\theta \in \lbrack1/2,1)$
	
	$$ \zeta_k(\theta,e) +   \overline{\Psi}_{a_{\lambda},b_{\mu} }\left( u_0(x+\frac{k}{2(1-\theta)}e)\right) $$
	with
	\begin{eqnarray}
	\int_{\R^N}&& \frac{x}{|x|} \left(|\nabla_A\left[u_0\left(x+ \frac{k}{2(1-\theta)}e\right)\right]|^2+ \left[u_0\left(x+ \frac{k}{2(1-\theta)}e \right) \right]^2 \right) \\
	\nonumber    &=&\int_{\R^N} \frac{x - \frac{k}{2(1-\theta)}e}{|x- \frac{k}{2(1-\theta)}e|} (|\nabla_A u_0(x)|^2+ u_0^2 )dx \\
	\nonumber    &=& \left( \frac{2p}{p-2} \right)m^{\infty}e+o_n(1) \;\;\; \mbox{if}\;\;\; \theta \rightarrow 1^-.
	\end{eqnarray}
	
	As by definition $    \overline{\Psi}_{a_{\lambda},b_{\mu} }(u) = \frac{  \int_{\R^N} \frac{x}{|x|}( |\nabla_A u|^2+u)^2)dx}{| \int_{\R^N} \frac{x}{|x|}( |\nabla_A u|^2+u)^2)dx|} $, it follows that $\lim_{\theta \rightarrow 1^-} \zeta_k(\theta,e)=e$.
	Moreover
	$$\lim_{\theta \rightarrow \frac{1}{2}^-} \zeta_k(\theta,e)= \gamma(s_1)= \overline{\Psi}_{a_{\lambda},b_{\mu} }(w_k)=  \overline{\Psi}_{a_{\lambda},b_{\mu} }(w_0(x+ke)) , $$
	since $\overline{\Psi}:\sum \rightarrow S^{N-1} $ is continuous.
	Thus we conclude that $ \zeta_k(\theta,e)\in C([0,1]\times S^{N-1}, S^{N-1} )$ and
	$$\zeta_k(0,e) = \Psi(\Phi(e)), \;\; \mbox{for all} \;\; e \in S^{N-1}$$
	$$\zeta_k(1,e) =e \;\; \mbox{for all} \;\; e \in S^{N-1}\;\; \mbox{and for all} \;\; k>k_0. $$
	Finally, we can conclude that $ \Psi \circ \Phi$ is homotopic to identity.
	
\end{proof}

\begin{lemma}\label{Lemma6.6}
	For each $\lambda \in (0,\lambda_0)$, $\mu \in (0,\mu_0)$ and $k>k^*$, the functional $ \jf $ has at least two critical points in $ \mathcal{Q}_{\lambda,\mu}=  [ \jf< m^+_{a_{\lambda},b_{\mu}}+m^{\infty }] $.
	
\end{lemma}

\begin{proof}
	From what we have seen in Lemma \ref{Lemma6.5}, the $\Psi \circ \Phi$ application is homotopic to the identity in $S^{N-1}$, for $\lambda \in (0,\lambda_0)$, $\mu \in (0,\mu_0)$ and $k>k^*$. Also, note that the domain of $ \Psi$ is equal to the image of $ \Phi $ and is given by the set we call $\mathcal{Q}_{\lambda,\mu}$. Thus, $cat(\mathcal{Q}_{\lambda,\mu})\geq 2$, fulfilling the hypothesis $ (ii) $ of Lemma \ref{Lemma6.2}.
	
	Still, under these conditions $\jf< m^+_{a_{\lambda},b_{\mu}}+m^{\infty } $ for all $u\in \mathcal{Q}_{\lambda,\mu} $ and by Proposition \ref{lemadecompacidade} if $\{u_n\}\subset M^-_{\lambda, \mu}$  is a minimizing sequence in $ \HA $ of $J_{\lambda,\mu} ,$ then there is a subsequence $\{u_n\}$ and $u_0 \in \HA$ non zero, such that $u_n=u_0+o_n(1)$ strong in $\HA$ and $J_{\lambda,\mu}(u_0)=\beta $, giving us that $\jf$ satisfies (PS) which fulfills the condition $ (i) $ of Lemma \ref{Lemma6.2}.
	Thus we can conclude that $ \jf $ has at least two critical points in $\mathcal{Q}_{\lambda,\mu}.$
	
\end{proof}

\begin{proof}[Proof of Theorem \ref{teo1.1(ii)}]
	For $\lambda \in (0,\lambda_0)$, $\mu \in (0,\mu_0)$, using the Theorem \ref{teo3.3} and Lemma \ref{Lemma6.6} we can conclude that problema $ \Plm $ has at least three solutions $\uma$, $u_1^-$ and $u_2^-$ with $\uma \in  M^+_{a_{\lambda},b_{\mu}}$, and $u_1^-$ and $u_2^- \;\; \in  M^-_{a_{\lambda},b_{\mu}}$. This concludes the proof of the Theorem \ref{teo1.1(ii)}.
	
\end{proof}

\section{The regularity of solutions to problems with $\Delta_A$}
In this section we will establish regularity results for the non-zero $\Plm$ with $\lambda > 0$ and $\mu > 0$.
Assuming that the conditions $ (A) $, $ (B_1) $ and $ (B_2) $ are satisfied, we will combine Brezis-Kato's regularity arguments \cite[Lemma B3]{struwe} and an argument similar to that used in \cite[Lemma 2.1]{ChabSzul}, to show that if $u$ is a solution of $ \Plm $, then  $u \in L^{\gamma}(\R^N)$ for all  $\gamma \in \lbrack 2^*,+\infty)$. 
To do this, we will need the results we will show below.

\begin{lemma}\label{B3}
	Consider
	$$h(x,u)=a_{\lambda}(x)|u|^{q-2}u+b_{\mu}(x)|u|^{p-2}u.$$
Assume that the conditions $(A)$, $(B_1) $ and $(B_2)$ are satisfied.
	then, there is $|v(x)| \in L^{\frac{N}{2}}$ such that
	$$  h(x,u)= v(x)(1+|u|)   .$$
\end{lemma}

\begin{proof}
	Notice that
	{\small
		\begin{eqnarray*}
			h(x,u)&=&a_{\lambda}(x)|u|^{q-2}u+b_{\mu}(x)|u|^{p-2}u\\
			&=&(1+|u|)\left(  \frac{a_{\lambda}(x)|u|^{q-2}u }{ 1+|u|}+ \frac{ b_{\mu}(x)|u|^{p-2}u }{ 1+|u|} \right).
	\end{eqnarray*}}
	Calling
	$$ v(x)= \left(  \frac{a_{\lambda}(x)|u|^{q-2}u }{ 1+|u|}+ \frac{ b_{\mu}(x)|u|^{p-2}u }{ 1+|u|} \right) ,$$
	we have
	{\small
		\begin{eqnarray*}
			|v(x)|&=& \left|  \frac{a_{\lambda}(x)|u|^{q-2}u }{ 1+|u|}+ \frac{ b_{\mu}(x)|u|^{p-2}u }{ 1+|u|} \right|\\
			&  \leq &\left|  \frac{a_{\lambda}(x)|u|^{q-2}u }{ 1+|u|}\right| + \left| \frac{ b_{\mu}(x)|u|^{p-2}u }{ 1+|u|} \right|\\
			&\leq &|a_{\lambda}(x)| |u|^{q-1}+\frac{|b_{\mu}(x)| |u|^{p-1}}{|u|}.
	\end{eqnarray*}}
	\normalsize
	Also,
	\begin{equation}\label{e.1}
	\int\left(  |a_{\lambda}(x)| |u|^{q-1}\right)^{N/2} \leq ||a_{\lambda}||_{q'}^{\frac{N}{2}} \left( \int |u|^{p'}  \right)^{\frac{1}{s'}},
	\end{equation}
	where $s= \frac{2p}{n(p-q)}$ and  $ p'=\frac{2p}{2p-n(p-q)}$. See that, $ p'=(q-1)\frac{np}{2p-n(p-q)} \leq  2^*$, whence $ |a_{\lambda}(x)| |u|^{q-1} \in L^{\frac{N}{2}} $.
	%fica $\frac{p(q-1)}{q}\leq 2^*$ that is, p vezes uma coisa menor q 1 'e menor que $2^*$
	Moreover,
	\begin{equation}\label{e.2}
	\int \left( \frac{|b_{\mu}(x)| |u|^{p-1}}{|u|}\right)^{\frac{N}{2}} \leq \int (  |b_{\mu}(x)| |u|^{p-2} )^{\frac{N}{2}} \leq C\int ( |u|^{p-2})^{\frac{N}{2}} .
	\end{equation}% pois g<1
	As $p<2^* $, we have $(p-2)\frac{N}{2} < (2^*-2)\frac{N}{2} = 2^*$ and by continuous embeding $\HA \hookrightarrow L^{\gamma},$ for all  $2 \leq \gamma \leq 2^*$, we have $ \displaystyle \int ( |u|^{p-2})^{\frac{N}{2}}<\infty$. So $ |b_{\mu}(x)| |u|^{p-2} \in L^{\frac{N}{2}} $. Thus, by (\ref{e.1}) and (\ref{e.2})
	$$  |h(x,u)|\leq |v(x)|(1+|u|) \;\;\;\mbox{with} \;\;\;|v(x)|\;\;\; \in L^{\frac{N}{2}}  .$$
\end{proof}

Now consider $\phi(x)= \eta(x)^2
u(x)\min\{|u(x)|^{\beta-1},L\}$, where $\beta>1$ and $L>0$ are constant and $u \in \HA$, and $\eta$ is a function $C^1$, bounded, such that its derivative is also bounded.

We denote $ \chi_{ \Omega}$ as the characteristic function of the set $ \Omega $. Then we will have
{\small
	\begin{eqnarray*}
		\overline{\nabla_A \phi} &=&     2\eta\nabla \eta \bar{u}\min\{|u(x)|^{\beta-1},L\}+\eta^2\overline{\nabla_A u}\min\{|u(x)|^{\beta-1},L\}  \\
		&  +& (\beta-1)\eta^2 \bar{u}  |u|^{\beta-2}\nabla |u|\chi_{|u|^{\beta-1}<L }
\end{eqnarray*}}
and
{\small
	\begin{eqnarray*}
		\nabla_A u \overline{\nabla_A \phi} &=& | \nabla_A u|^2\eta^2\min\{|u|^{\beta-1},L\}+    2\eta\nabla \eta \bar{u}\min\{|u|^{\beta-1},L\}\nabla_A u  \\
		&  +& (\beta-1)\eta^2 \bar{u}  |u|^{\beta-2}\nabla |u|\chi_{|u|^{\beta-1}<L }\nabla_A u.
\end{eqnarray*}}
Note that
$$ Re(\bar{u}\nabla_A u) = Re( \nabla u +iAu)\bar{u}= Re (\bar{u}\nabla u )= |u|Re( \frac{\bar{u}}{|u|}\nabla u)=|u|\nabla|u|. $$
Using the real part of $ \nabla_A u \overline{\nabla_A \phi} $ we get
{\small
	\begin{eqnarray*}
		Re(\nabla_A u \overline{\nabla_A \phi}) &=& | \nabla_A u|^2\eta^2\min\{|u|^{\beta-1},L\}+    2\eta\nabla \eta \nabla |u||u|\min\{|u|^{\beta-1},L\}  \\
		&  +& (\beta-1)\eta^2    |u|^{\beta-1}|\nabla |u||^2\chi_{|u|^{\beta-1}<L }\\
		&\leq&     | \nabla_A u|^2\eta^2\min\{|u|^{\beta-1},L\}+    2\eta\nabla \eta \nabla |u||u|\min\{|u|^{\beta-1},L\} .
\end{eqnarray*}}
So
\small{
	\begin{equation}\label{preal}
	Re(\nabla_A u \overline{\nabla_A \phi})\leq     | \nabla_A u|^2\eta^2\min\{|u|^{\beta-1},L\}+    2\eta\nabla \eta \nabla |u||u|\min\{|u|^{\beta-1},L\} .
	\end{equation}}
\begin{lemma}\label{lemachabr}
	The $\Plm$ solutions with $\lambda > 0$ and $\mu > 0$, in $\HA$, belong to $L^{\gamma}(\R^N)$ for all  $\gamma\in \lbrack2^*,+\infty).$
\end{lemma}

\begin{proof}
	We will test our problem  with the function $\phi=u\min\{|u|^{\beta-1},L\} $. Note que $ (u\phi)^{\frac{1}{2}}= u  \min\{|u|^{\beta-1},L\}^{\frac{1}{2}} $. Our first objective is to show that
	$$ || |u| \min\{|u|^{\beta-1},L\}^{\frac{1}{2}}||^2_{2^*} \leq C\int |u|^2\min\{|u|^{\beta-1},L\}  .  $$
	
	For this, we will see in parts, the following sequence of four inequalities
	{\small
		\begin{eqnarray}
		\label{1.1}  || |u|\min\{|u|^{\beta-1},L \}^{\frac{1}{2}}||^2_{2^*} &\leq &  \int | \nabla(|u|\min\{|u|^{\frac{\beta-1}{2}},L^{\frac{1}{2}}\})|^2dx\\
		\label{1.2}  &\leq & C(\beta)\int |\nabla|u||^2\min\{|u|^{ \beta-1} ,L \}dx\\
		\label{1.3}  &  \leq &  C(\beta) \int |\nabla_A u|^2\min(|u|^{ \beta-1},L  ))dx\\
		\label{1.4} &  \leq &  C(K,\beta)\int |u|^2 \min\{|u|^{ \beta-1} ,L\}dx.
		\end{eqnarray}}	
The first inequality (\ref{1.1}), we derive from the Sobolev inequality. Indeed
	$$  || |u|\min\{|u|^{\beta-1},L\}^{\frac{1}{2}}||^2_{2^*} \leq   |||u|\min\{|u|^{\beta-1},L\}^{\frac{1}{2}}||^2_{H_0^1} .$$	
	Also, to verify (\ref{1.2}) we observe that
	$$|||u|\min\{|u|^{\beta-1},L\}^{\frac{1}{2}}||^2_{H_0^1} = \int | \nabla(|u|\min\{|u|^{\frac{\beta-1}{2}},L^{\frac{1}{2}}\})|^2 + |(u\min\{|u|^{\frac{\beta-1}{2}},L^{\frac{1}{2}}\})|^2dx  ,$$
	so	{\small
		\begin{eqnarray*}
			\lefteqn{ \int  | \nabla(|u|\min(|u|^{\frac{\beta-1}{2}},L^{\frac{1}{2}}  ))|^2+ [(|u|\min\{|u|^{\frac{\beta-1}{2}},L^{\frac{1}{2}}\})]^2dx}\\
			& \leq & \int 2(\nabla |u|)^2\min\{|u|^{\beta-1},L\}+ 2 |u|^2\left(u\nabla u \frac{\beta -1}{2} |u|^{\frac{\beta -1}{2}-2} \chi_{|u|^{\frac{\beta -1}{2}}<L^{\frac{1}{2}}}\right)^2\\ &+&[(|u|\min\{|u|^{\frac{\beta-1}{2}},L^{\frac{1}{2}}\})]^2\\
			% using that $(u\nabla u)^2=u^2|\nabla|u||^2$ mas so sei que $|u|\nabla |u|u\nabla u=(u\nabla u)^2$
			&  \leq &\left( 2+ \frac{(\beta -1)^2}{2}\right) \int |\nabla |u||^2\min\{|u|^{\beta-1},L\}+ [(|u|\min\{|u|^{\frac{\beta-1}{2}},L^{\frac{1}{2}}\})]^2.
	\end{eqnarray*}}	
The third inequality (\ref{1.3}) we obtain from the diamagnetic inequality
	{\small
		\begin{eqnarray*}
			C\int |\nabla |u||^2\min\{|u|^{\beta-1},L\}\leq C \int |\nabla_Au|^2\min\{|u|^{\beta-1},L\}.
	\end{eqnarray*}}	
	Finally, it follows from inequality(\ref{preal}), with $\eta =1$, that for every constant $K$%=1$%Chabrowisky (2.2) pag3 e caderno HWW pag 5
	{\small
		\begin{eqnarray*}
			\lefteqn{  \int_{\R^N}   | \nabla_A u|^2 \min\{|u|^{\beta-1},L\}   +   |u|^2\min\{|u|^{\beta-1},L\}\leq \int  Re(\nabla_A u \overline{\nabla_A \phi})+ Re|u|^2\min\{|u|^{\beta-1},L\}}\\
			&=&Re \int( -\Delta_Au\overline{\phi}+u\overline{\phi} ) \leq  \int| h(x,u) || \overline{\phi} | \leq \int |v(x)|(1+|u|)|\phi|\\	
			& \leq&  \left(\int_{|v(x)|>K} |v(x)|^{\frac{N}{2}}\right)^{\frac{2}{N}}\left(\int(|u|\min\{|u|^{\beta-1},L\})^{\frac{N}{N-2}}\right)^{\frac{N-2}{N}}+ K\int_{|v(x)|<K}|u|\min\{|u|^{\beta-1},L\}\\
			&+&\left(\int_{|v(x)|>K}v(x)^{\frac{N}{2}}\right)^{\frac{2}{N}}\left(\int(|u|^2\min\{|u|^{\beta-1},L\})^{\frac{N}{N-2}}\right)^{\frac{N-2}{N}}+ K\int_{|v(x)|<K}|u|^2\min\{|u|^{\beta-1},L\}.
	\end{eqnarray*}}
For $ K $ large enough we have $\left(\int_{|v(x)|>K}v(x)^{\frac{N}{2}}\right)^{\frac{2}{N}}=o_n(1)$, since by Lemma \ref{B3}, $|v(x)| \in L^{\frac{N}{2}}$. Still, making $ L \rightarrow \infty $ and taking $\beta +1=2^*$ we get
	{\small
		\begin{eqnarray*}
			\int_{\R^N}   | \nabla_A u|^2 \min\{|u|^{\beta-1},L\}  &+ & |u|^2\min\{|u|^{\beta-1},L\}\leq \\
			&\leq &   K\int|u|^{2^*-1}+K\int|u|^{2^*}+o_N(1)<\infty,
	\end{eqnarray*}}
	as $u \in \HA \hookrightarrow L^p$ for $2<p\leq2^*$.
	
	We conclude that $u_0$ is a solution of $\Plm$, then $u \in L^{\gamma}(\R^N)$ for all  $\gamma\in \lbrack 2^*,\infty).$
\end{proof}

\begin{theorem}\label{regul}
	Supose that $u_0 \in \HA$ is a $\Plm$ non-zero solution with $\lambda > 0$ and $\mu > 0$. Then, we have $u_0 \in C(R^N,\C) \cap L^{\gamma}(R^N)$ for all $2^* \leq \gamma  < +\infty$ .
\end{theorem}

\begin{proof}
	By Lemma \ref{lemachabr}, as $u_0 \in \HA$ is a $\Plm$ non-zero solution, so he
	belong to $L^{\gamma}(\R^N)$ for all  $\gamma\in \lbrack2^*,+\infty).$
	
	Now, to show the other part of the intersection and apply the theory of regularity we need to separate our problem into two others by doing
	$$ u=v+iw.$$
	See that,
	{\small
		\begin{eqnarray*}
			-\Delta_Au+u&=&-\Delta u-2iA\nabla u+|A|^2u-iu \mbox{ div } A +u\\
			&=& a_{\lambda}|u|^{q-2}u+b_{\mu}|u|^{p-2}u,
	\end{eqnarray*}}
	 so 
	$$ -\Delta v + v = h_1 =  a_{\lambda}|u|^{q-2}v+b_{\mu}|u|^{p-2}v -  2A\nabla w - |A|^2v -w \,\mbox{div} A,$$
	$$ -\Delta w + w = h_2 =  a_{\lambda}|u|^{q-2}w + b_{\mu}|u|^{p-2}w +  2A\nabla w - |A|^2w + w\, \mbox{div} A.$$
	With this we have $ v , w \in H^1 \hookrightarrow L^{2^*}$ and by \cite[Theorem 1.9]{willen}, follow that $ h_1    $ and $h_2$ $\in L^{q_1}(K,\R),$ where $q_1=\min\{2^*(p-1)^{-1},2\}$.
	By a standard argument $v,w \in W^{2,q_1}$. Using \cite[Corolary 9.13]{Brezis}, if $2q_1<N$ we get, $v,w \in L^{\frac{Nq_1}{N-2q_1}}$ and still $\nabla v,\nabla w \in L^{\frac{Nq_1}{N-2q_1}}$.
	Again, doing the same calculations we get $ h_1    $ and $h_2$ $\in L^{q_2}(K,\R),$ where $q_2=Nq_1\min\{(N-2q_1)(p-1)^{-1},(N-q_1)^{-1}\}$.
	We now use the boot-strap argument to conclude that after a finite number of steps we will have $ v,w \in W^{2,q} $ for all  $q \in \lbrack 1, +\infty)$ and by Sobolev embedding, $ v $ and $w \in C^{1,\alpha}(K,\R)$ with $0<\alpha <1$, whence we get the desired result for $u.$

	%	(ii) Vamos mostrar agora que $|u_0|$ é positiva em $\R^N$. 
	%Usaremos um resultado de \cite[Teorema 2.5]{GT} de Trundinger página 16. 	De fato, suponhamos por absurdo que exista um conjunto $\mathcal{C}$ de medida não nula such that $|u_0|=0$ em $\mathcal{C}$. 	Assim, for all  conjunto $\Re '$ such that $ \mathcal{C} \subset \Re ' \subset \subset \R^N$ vai existir $x_0 \in  \mathcal{C}$   such that $ |u_0(x_0)|=0$. Pela desigualdade de Harnack \ref{minmax}, vai existir um $c$ such that \begin{equation}\label{harnack}
	%	\sup_{\Re '}|u_0| \leq c\inf_{\Re '}|u_0| . 
	%	\end{equation} 
	
	%	Como $\mathcal{C} \subset \Re '$, follow that $\inf_{\Re '}|u_0|=0$. Dessa forma temos $\sup_{\Re '}|u_0| = 0 $. Por isso e por (\ref{harnack}) obtemos $|u_0|\equiv 0$ o que é um absurdo pois por hipótese tomamos uma solução não nula.
\end{proof}

To conclude the regularity results we present the following lemma.
\begin{theorem}\label{limitezero}
	If $u\in \HA$ is $ \Plm $ solution, then $u \in L^{\infty}(\R^N)$ and $\lim_{|x|\rightarrow \infty}u(x)=0.$
\end{theorem}
\begin{proof}
	We will use the Moser's interaction technique.
	Let $ \eta $ be a compact $ C^1 $ function. Rewriting the problem $ \Plm $ as follows
	{\small
		\begin{eqnarray}\label{prob}
		\nabla_Au+u=g(x,|u|)u,
		\end{eqnarray}}
	we get $g(x,|u|)u= v(x)(1+|u|)$, with $v \in L^{N/2}$, as has already been seen in Lemma \ref{B3}. Testing (\ref{prob}) with $\phi=\eta^2u \min\{|u|^{\beta -1},L\}$ and using the inequality ($ \ref{preal} $) we get the estimate
	{\small
		\begin{eqnarray*}
			\lefteqn{ \int_{\R^N} (  | \nabla_A u|^2 \eta^2 \min\{|u|^{\beta-1},L\}   + 2\eta \nabla \eta  \nabla|u||u|\min\{|u|^{\beta-1},L\} dx\leq} \\
			&\leq& \int  Re(\nabla_A u \overline{\nabla_A \phi}) =Re \int  -\Delta_Au\overline{\phi} = Re \int  (-u+g(x)u) \overline{\phi}\\
			&\leq &\int|  g(x)  u\overline{\phi}|+|u \overline{\phi}|\leq \int|  v(x)||\overline{\phi}|+  |v(x)||u|)||\overline{\phi}|. 
	\end{eqnarray*}}
	Note that 
	{\small
		\begin{eqnarray*}
			\frac{1}{2}\eta^2  | \nabla| u||^2-2|u|^2|\nabla\eta|^2 &\leq &\eta^2 |\nabla|u||^2+2 \eta|u| \nabla|u|\nabla \eta \\
			&\leq &\eta^2 |\nabla_Au|^2+2 \eta|u| \nabla|u|\nabla \eta.
	\end{eqnarray*}}
	\normalsize
	In this way,
	{\small
		\begin{eqnarray*}
			%\lefteqn{ }
			\frac{1}{2}\int_{\R^N}  | \nabla| u||^2\eta^2 \min\{|u|^{\beta-1},L\}  &\leq&  \int_{\R^N} (|\nabla_Au|^2\eta^2\min\{|u|^{\beta-1},L\}+
			2 \eta|u| \nabla|u|\nabla \eta \min\{|u|^{\beta-1},L\}\\
			&&+2|u|^2|\nabla\eta|^2\min\{|u|^{\beta-1},L\})\\
			&\leq &\int_{\R^N} (  |  v(x)||\overline{\phi}|+  |v(x)||u|)||\overline{\phi}|    +2|u|^2|\nabla\eta|^2\min\{|u|^{\beta-1},L\}.
	\end{eqnarray*}}
	\normalsize
	Doing $L \rightarrow \infty$ 
	{\small
		\begin{eqnarray}\label{meio}
		\nonumber			\lefteqn{ \frac{1}{2}\int_{\R^N}  | \nabla| u||^2\eta^2 |u|^{\beta-1} \leq  2\int_{\R^N} |\nabla\eta|^2|u|^{\beta+1}+\int |v(x)||1+|u|||u|^{\beta}\eta^2}\\
		&= &2\int_{\R^N} |\nabla \eta|^2|u|^{\beta+1}+C\int |v(x)||u|^{\beta}\eta^2+C\int |v(x)||u|^{\beta+1}\eta^2.
		\end{eqnarray}}
	\normalsize
Still doing $\omega = |u|^{\frac{\beta +1}{2}}$  
	{\small
		\begin{eqnarray}\label{1.5}
		\frac{2}{(\beta +1)^2}\int_{\R^N}  | \nabla \omega|^2\eta^2   &\leq &2\int_{\R^N} |\nabla \eta|^2\omega^2+C\int |v(x)|\omega^2\eta^2+C\int |v(x)|\omega^{\frac{2\beta}{\beta+1}}\eta^2.
		\end{eqnarray}}
	\normalsize
	% so   {\small \begin{eqnarray*}   (i)&=&\int_{\R^N} |\nabla\eta|^2\omega^2\\  (ii)&=&\int |v(x)|\omega^2\eta^2\\  (iii)&= & \int |v(x)|\omega^{\frac{2\beta}{\beta+1}}\eta^2 \end{eqnarray*}}
	For the last two terms of (\ref{1.5}) we have
	%$(ii):$
	{\small
		\begin{eqnarray*}
			\int |v(x)|\omega^2\eta^2   &\leq &\left(\int_{\R^N} |v(x)|^{\frac{N}{2}}\right)^{\frac{2}{N}} \left(\int_{\R^N}(\eta \omega)^{2^*} \right)^{\frac{N-2}{N}},
	\end{eqnarray*}}\normalsize
	and also
	%$(iii):$
	{\small
		\begin{eqnarray*}
			\int |v(x)|\omega^{\frac{2\beta}{\beta+1}}\eta^2 &= &\int |v(x)|(\omega\eta)^{\frac{2\beta}{\beta+1}}\eta^{\frac{2}{\beta+1}}\\
			&\leq &||v||_{\frac{N}{2}} ||\omega \eta||_{2^*}^{\beta \frac{N-2}{N}}||\eta||_{2^*}^{\frac{N-2}{N}}.
	\end{eqnarray*}}
	\normalsize
	Now, to show the integrability of the first term of inequality (\ref{1.5}) note that
	{\small
		\begin{eqnarray*}
			\int_{\R^N}  | \nabla( \omega \eta)|^2 &\leq &2\int_{\R^N} |\nabla\omega|^2\eta^2+2\int|\nabla\eta|^2\omega^2
	\end{eqnarray*}}\normalsize
which together with (\ref{meio}) gives us
	{\small
		\begin{eqnarray*}
			\int_{\R^N}  | \nabla( \omega \eta)|^2 &\leq &C(\beta-1)^2\int_{\R^N} a\omega^2\eta^2  +  2((\beta-1)^2-1)\int|\nabla\eta|^2\omega^2.
	\end{eqnarray*}}
	\normalsize
	By Sobolev and Holder
	{\small
		\begin{eqnarray*}
			S^2\left(\int_{\R^N} (\omega \eta)^{2^*}\right)^{\frac{N-2}{2}}&\leq& C (\beta -1)^2\left(\int_{\R^N}|v(x)|^\frac{N}{2}  \right)^{\frac{ 2}{N}}\left(\int_{\R^N} (\omega \eta)^{2^*}\right)^{\frac{N-2}{2}}\\
			&+& ((\beta-1)^2-1)   \int_{\R^N} (|\nabla\eta|^2\omega^2)+(\beta-1)^2\int |v(x)|\omega^{\frac{2\beta}{\beta+1}}\eta^2, 
	\end{eqnarray*}}\normalsize
	on what $S=\inf\{\int_{\R^N}|\nabla u|^2dx; u \in C_0^\infty(\R^N); \;\int_{\R^N}|u|^{2^*}dx=1\}$ is the Sobolev constant. Still, we chose a convenient $ R> 0 $ so that
	$$ C(\beta-1)^2\left(\int_{|x|>R} |v(x)|^\frac{N}{2}\right)^{\frac{2}{N}} \leq \frac{S^2}{2} .  $$
	Now, assuming that $\mbox{supp}\eta\subset (|x|>R)  $ we have
	$$\int_{|x|<R}(\omega\eta)^{2^*}=0,$$
	 so,
	{\small
		\begin{eqnarray}\label{4a}
		S^2\left(\int_{\R^N} (\omega \eta)^{2^*}\right)^{\frac{N-2}{2}}&\leq& 4 ((\beta-1)^2-1)   \int_{\R^N} |\nabla\eta|^2\omega^2.
		\end{eqnarray}}
	\normalsize
	To proceed with the interaction we will take a $\eta \in C^1(\R^N,[0,1])$, with $\eta(x)=1$ in $B(x_0,r_1),$ $\eta(x)=0 $ in $ \R^N \setminus \{B(x_0,r_1)\} $ and already $ |\nabla\eta(x)|\leq \frac{2}{r_2-r_1}$ in $\R$, where $1\leq r_1<r_2\leq2.$ Moreover, we choose $x_0$ and $r_2$ so that $B(x_0,r_2) \subset (|x|>R)$.
	With this, it follows from (\ref{4a}) that
	{\small
		\begin{eqnarray*}\label{4}
			\left(\int_{ B(x_0,r_1)} \omega  ^{2^*}\right)^{\frac{1}{2^*}}&\leq&\frac{2}{S} ((\beta+1)^2+1)^{\frac{1}{2}}\frac{2}{r_2-r_1}\left(\int_{ B(x_0,r_1) }  \omega^2\right)^{\frac{1}{2}} \\
			&\leq&T\frac{\beta+1}{r_2-r_1}\left(\int_{ B(x_0,r_1) }  \omega^2\right)^{\frac{1}{2}} 
	\end{eqnarray*}}\normalsize
	with T an absolute constant. Making $\gamma =\beta +1=2^*$ and $  \chi=\frac{N}{N-2}$ we get
	{\small
		\begin{eqnarray*}
			\left(\int_{B(x_0,r_1)} |u|^{\gamma \chi}\right)^{\frac{1}{\gamma \chi}}&\leq&   \left(\frac{T\gamma}{r_2-r_1}\right)^{\frac{2}{\gamma }}+ \left(\int_{B(x_0,r_1)} |u|^{\gamma  }\right)^{\frac{1}{\gamma  }}.
	\end{eqnarray*}}\normalsize
	To iterate the inequality, which follows with $ \gamma \geq 2^* $, we take $s_m=1+2^{-m}$, $r_1=s_m,$ $r_2=s_{m-1}$ and we replaced $\gamma = 2^*$ by $\gamma \chi^{m-1}$, to $m=1,2,...$. In this way we obtain
	{\small
		\begin{eqnarray*}
			\left(\int_{B(x_0,s_m)} |u|^{\gamma \chi^m}\right)^{\frac{1}{\gamma \chi^m}}&\leq& \left(\frac{T\gamma\chi^{m-1}}{s_{m-1}-s_m}\right)^{\frac{2}{\gamma\chi^{m-1}}}
			\left(\int_{B(x_0,s_{m-1})} |u|^{\gamma\chi^{m-1}}\right)^{\frac{1}{\gamma \chi^{m-1}}}\\
			& = &(T\gamma)^{\frac{2}{\gamma \chi^{m-1}}}
			2^{\frac{2m}{\gamma \chi^{m-1}}}
			\chi^{\frac{2(m-1)}{\gamma\chi^{m-1}}}
			\left(\int_{B(x_0,s_{m-1})} |u|^{\gamma\chi^{m-1}}\right)^{\frac{1}{\gamma \chi^{m-1} }}.
	\end{eqnarray*}}\normalsize
	By induction
	{\small
		\begin{eqnarray*}
			\left(\int_{B(x_0,s_m)} |u|^{\gamma \chi^m}\right)^{\frac{1}{\gamma \chi^m}}&\leq&
			(T\gamma)^{\frac{2}{\gamma}\sum_{j=0}^{m-1}\frac{1}{\chi^j}}
			2^{\frac{2m}{\gamma }\sum_{j=0}^{m-1}\frac{j+1}{\chi^j}}
			\chi^{\frac{2}{\gamma}\sum_{j=0}^{m-1}\frac{j}{\chi^j}}
			\left(\int_{B(x_0,s_0)} |u|^{\gamma }\right)^{\frac{1}{\gamma  }}
	\end{eqnarray*}}\normalsize
	for all  $m \in \N $.
	As $s_0=2$ and as $m\rightarrow \infty$ we get $s_m\rightarrow 1$, we deduce by doing $m\rightarrow \infty$ that exists $R>0$ and $C>0$ such that for all $B(x_0,2)\subset(|x|>R)$ 
	% {\small\begin{eqnarray*}  \left(\int_{B(x_0,1)} |u|^{\gamma \chi^{\infty}}\right)^{\frac{1}{\gamma \chi^{\infty}}}\\\end{eqnarray*}}
	{\small
		\begin{eqnarray*}
			\sup_{B(x_0,1)}|u(x)|&\leq&C \left(\int_{B(x_0,2)} |u|^{\gamma }\right)^{\frac{1}{\gamma  }}.
	\end{eqnarray*}}\normalsize
	As  $\int |u|^{\gamma } < \infty $, then $\int_{B_{(x_i,R)}^C} |u|^{\gamma } \rightarrow 0$ if $ R\rightarrow \infty$ whence we get $u(x) \rightarrow 0 $ as $|x| \rightarrow \infty .$
	
To prove the limitation of $ u$ on the ball $ B (0, R) $ we set $ \bar {x} \in B (0, R) $ and we choose $r>0$ such that
	$$ (\beta+1)^2\left(\int_{B(\bar{x}, r)}b^{\frac{N}{2}}\right)^{\frac{2}{N}}\leq \frac{S}{2}$$
	\noindent and supose $\eta$ with support in $B(\bar{x}, r.)$
We then repeat the previous argument with an appropriate rescaling on $ B (\bar {x}, r) $ to obtain the $ u$ limit on $B(\bar{x}, \frac{r}{2})$.
	Being $\overline{B(0,R)}$ compact, and since $u $ is bounded in each compact $ B (x, \frac{r}{2}) $ for each $ x \in B (0, R) $, we conclude that  $\overline{B(0,R)}$ has finite subcoverage of compacts giving us that $ u$ is bounded in $B(0,R).$
	With this and the fact shown in the first part of this theorem we obtain $u\in L^{\infty}$.
\end{proof}

\end{document}